\documentclass[a4paper]{amsart}
\usepackage{amssymb
	,amsthm,dirtytalk
	,amsmath,color,todonotes
	,amscd,mathrsfs
	,mathtools,enumerate
}
 \usepackage{enumitem}

\usepackage{hyperref}
\usepackage[all]{xy}
\usepackage[top=30truemm,bottom=30truemm,left=25truemm,right=25truemm]{geometry}
\newcommand{\St}{\mathrm{St}}
\newcommand{\Sp}{\mathrm{Sp}}

\newcommand{\SL}{\mathrm{SL}}
\newcommand{\GL}{\mathrm{GL}}

\newcommand{\Ind}{\mathrm{Ind}}
\newcommand{\Res}{\mathrm{Res}}
\newcommand{\Hom}{\mathrm{Hom}}

\newcommand{\Ext}{\mathrm{Ext}}

\newcommand{\Nm}{\mathrm{Nm}}

\newtheorem{thm}{Theorem}[section]

\newtheorem{lem}[thm]{Lemma}
\newtheorem{prop}[thm]{Proposition}
\newtheorem{coro}[thm]{Corollary}

\theoremstyle{remark}

\theoremstyle{definition}

\newtheorem{rem}[thm]{Remark}

\numberwithin{equation}{section}

\makeatletter
\def\iddots{\mathinner{\mkern1mu\raise\p@
		\hbox{.}\mkern2mu\raise4\p@\hbox{.}\mkern2mu
		\raise7\p@\vbox{\kern7\p@\hbox{.}}\mkern1mu}}
\def\adots{\mathinner{\mkern2mu\raise\p@\hbox{.}
		\mkern2mu\raise4\p@\hbox{.}\mkern1mu
		\raise7\p@\vbox{\kern7\p@\hbox{.}}\mkern1mu}}
\makeatother

\allowdisplaybreaks
\title{The modulo distinction problem for $D^\times$}
\author{Junyan Liu}
\address{School of Mathematical Sciences, Beihang University, 9 Nansan Street, Shahe Higher Education Park, Changping, Beijing, 102206, P.R.C.}
\email{mayljy@buaa.edu.cn}
\author{Hengfei LU}

\email{luhengfei@buaa.edu.cn}

\date{\today}

\begin{document}
	
	\begin{abstract} Let $F$ be a nonarchimedean local field of characteristic zero and residue characteristic $p$. Let $E$ be a quadratic field extension of $F$. Let $D$ be the unique quaternion division algebra defined over $F$. Let $G=\GL_{2}(E)$  and $H=D^\times$.
		This paper studies the modulo distinction problem for the pair $(G,H)$ with $\ell\neq p$, which gives a complete classification of $H$-distinguished representations of $G$, including the cases $\ell=2$ and $p=2$. In terms of type theory, we give a complete classification of $H$-distinguished supercuspidal representations when $p$ is odd.
	\end{abstract}
	\maketitle
	\tableofcontents
	\section{Introduction}
Let \( F \) be a non-archimedean local field of characteristic zero and residue characteristic \( p \). Let \( E/F \) be a quadratic field extension. Let \( D \) be the unique quaternion division algebra over \( F \). There is a natural embedding \( D^\times \hookrightarrow \GL_2(E) \), so \( D^\times \) can be viewed as a closed subgroup of \( \GL_2(E) \). Let \( \ell \) be a prime different from \( p \). We say that an irreducible smooth \( \overline{\boldsymbol{F}_\ell} \)-representation \( \pi \) of \( \GL_2(E) \) is \( D^\times \)-distinguished, or simply distinguished if
\(\Hom_{D^\times}(\pi, \boldsymbol{1}) \neq 0,\) where \( \boldsymbol{1} \) denotes the trivial character of \( D^\times \).

Distinction problems are central objects in the relative Langlands program. Over the complex field, the study of distinction on the symmetric varieties has a rich history. Prasad's work \cite{prased1992} established a fundamental connection between \( D^\times \)-distinction and \( \GL_2(F) \)-distinction for supercuspidal representations. This theme has been explored further in various settings; for instance, the period problems of \( \GL_2 \) over quadratic extensions have been revisited using theta lifts and Whittaker model (see \cite{Lu_2017}). More recently, Choiy and Patel \cite{choiy2025} have investigated the analogous distinction problem for \( \SL_n(D) \) with respect to the subgroup \( \GL_n(E)\cap \SL_n(D) \).

In the modular setting, the theory of \( \ell \)-modular representations of \( p \)-adic groups has been developed systematically by Vign\'eras \cite{vig1989,vig1998} and distinction problems for modular representations have recently attracted considerable interest, such as \cite{vincent2019,zoujiandi22bullsmf,robinIMRN,vin2025,vin2026,zhangrobin26}. For the pair \( (\GL_2(E), \GL_2(F)) \), a complete classification of distinguished representations was obtained in \cite{cui2024}, but the analogous problem for the group \( D^\times \) was left open.

\subsection{}Let \(\chi_1,\chi_2\) be characters of \(E^\times\). For an irreducible principal series representation \(\pi(\chi_1, \chi_2)\) of \(\GL_2(E)\), a standard Mackey-theoretic argument using the double coset decomposition \(\mathrm{GL}_2(E) = B_2(E)D^\times\) shows that it is \(D^\times\)-distinguished if and only if
\(\chi_1\chi_2^\sigma = 1,\) and in that case the corresponding Hom space is one-dimensional (see Lemma \ref{3.2}). 

The more delicate part is to determine which of the irreducible subquotients of a reducible principal series are distinguished. This requires a detailed study of the relevant Ext groups. A key technical step is the computation of \(\Ext_{D^\times}^1(\boldsymbol{1}, \chi \circ \det)\) for characters \(\chi\) of \(F^\times\), which we carry out mainly via the Hochschild--Serre spectral sequence (see Lemma \ref{3.3}). This computation is particularly subtle when \(\chi\) is quadratic and \(\ell\) divides \(q+1\), and it is also responsible for the exceptional behavior in the cases \(p=2\) or \(\ell=2\). Based on this, we obtain the following complete classification for the non-supercuspidal subquotients distinguished by $D^\times$.

\begin{thm}[Theorem \ref{thm1}, Proposition \ref{pSp2} and Theorem \ref{thm41}]\label{thm:principal}
Suppose $\ell\neq p$. Let \(\chi\) be a character of \(E^\times\).\label{mainthm1.1}
\begin{enumerate}
\item If \(\ell \neq 2\) and \(q_E \not\equiv -1 \pmod{\ell}\), then the twisted Steinberg representation \(\mathrm{St}_\chi\) is \(D^\times\)-distinguished if and only if \(\chi|_{F^\times} = \omega_{E/F}\).

\item If \(\ell \neq 2\) and \(q_E \equiv -1 \pmod{\ell}\), then the special representation \(\mathrm{Sp}_\chi\) is \(D^\times\)-distinguished if and only if one of the following conditions holds:
    \begin{enumerate}
        \item $\chi|_{F^\times}=\omega_{E/F}$ and $E/F$ is ramified;
        \item $\chi|_{F^\times}$ is a nontrivial quadratic unramified character and $E/F$ is ramified.
    \end{enumerate}

\item If \(\ell = 2\), then $\Sp_\chi$ is distinguished by $D^\times$ if and only if $\chi|_{F^\times}=\boldsymbol{1}$, $E/F$ is ramified and $q\equiv 3\pmod 4$.
\end{enumerate}

Moreover, in all positive cases, the dimension of the corresponding \(\mathrm{Hom}\) space is one.
\end{thm}
\begin{rem}
This theorem exhibits two phenomena specific to the modular setting.
\begin{itemize}
\item Theorem \ref{thm:principal} (2) condition (b) implies that an irreducible subquotient $\Sp_\chi$ of a principal series representation can be \(D^\times\)-distinguished even when the principal series itself is not. 
Thus distinction of the special representation cannot be detected only from invariant linear forms on the principal series. In our proof, this additional invariant linear form is accounted for by a nonzero extension group;
\item Theorem \ref{thm:principal} (3) shows that when \(\ell=2\), the distinction criterion detects the residue cardinality modulo \(4\). This dependence arises from the extension structure of a finite permutation representation of the dihedral group \(D_{2(q+1)}\).
\end{itemize}
\end{rem}

\subsection{}For supercuspidal representations, we reduce the \( D^\times \)-distinction problem to the more familiar \( \GL_2(F) \)-distinction problem. The key ingredient is a lifting theorem: if a supercuspidal \( \overline{\boldsymbol{F}_\ell} \)-representation is \( D^\times \)-distinguished, then it admits a \( D^\times \)-distinguished \( \overline{\mathbb{Q}_\ell} \)-lift. Combining this with Prasad's theorem \cite{prased1992} in the complex case yields the following result. 

\begin{thm}[Theorem \ref{c3.1}]\label{thm:equiv}
Let \( \pi \) be an irreducible supercuspidal representation of \( \mathrm{GL}_2(E) \). Then \( \pi \) is \( D^\times \)-distinguished if and only if \( \pi \) is \( \mathrm{GL}_2(F) \)-distinguished.
\end{thm}

Thus, for supercuspidal representations, the \( D^\times \)-distinction problem is equivalent to the \( \mathrm{GL}_2(F) \)-distinction problem. For \( \ell \neq 2 \), the latter is governed by the following dichotomy \cite{cui2024}.

\begin{thm}[Theorem \ref{thm42}] Assume $\ell\neq2$.
Let \( \pi \) be a \( \sigma \)-selfdual supercuspidal \( \overline{\boldsymbol{F}_\ell} \)-representation of \( \mathrm{GL}_2(E) \). Then the following are equivalent:
\begin{enumerate}
\item \( \pi \) is \( \mathrm{GL}_2(F) \)-distinguished;
\item \( \pi \) is not \( (\mathrm{GL}_2(F), \omega_{E/F}) \)-distinguished;
\item the Langlands parameter \( \phi_\pi \) is conjugate-orthogonal.
\end{enumerate}
In particular, a \( \sigma \)-selfdual supercuspidal representation is either \( \mathrm{GL}_2(F) \)-distinguished or \( (\mathrm{GL}_2(F), \omega_{E/F}) \)-distinguished, but not both.
\end{thm}

For the classification of supercuspidal representations, we may rely on the theory of types, which provides an explicit construction of these representations via compact induction. In the complex case, this theory was developed by Bushnell  for \( \GL_n(F) \) (see \cite{B.K1993,B.H.2006}). This framework has been successfully extended to the \( \ell \)-modular setting. Using the type-theoretic construction of S\'echerre \cite{vincent2019,vin2025,vin2026}, one can make the criterion for the distinction completely explicit. Attached to a \( \sigma \)-selfdual supercuspidal representation \( \pi \) is its distinction character \( \delta_0 \), which is either trivial or, when \( \ell \neq 2 \), the quadratic character \( \omega_{\mathcal D/\mathcal D_0} \) associated to a certain quadratic extension \( \mathcal D/\mathcal D_0 \) arising from a \( \sigma \)-selfdual stratum.

\begin{thm}[Theorem \ref{sccriterion}]
Let $\pi$ be a $\sigma$-selfdual supercuspidal representation of $\GL_2(E)$,
with relative degree $m \in \{1,2\}$ and distinction character $\delta_0$.
\begin{enumerate}
\item If $\ell = 2$, then $\pi$ is always $D^\times$-distinguished.
\item If $\ell \neq 2$ and $p$ is odd, then $\pi$ is $D^\times$-distinguished if and only if 
\begin{align*}
    \begin{cases}
    \delta_0\ \mathrm{is\ the\ trivial\ character},&\mbox{when}\ m=1;\\\
    \delta_0=\omega_{\mathcal{D}/\mathcal{D}_0},&\mbox{when}\ m=2.
\end{cases}
\end{align*}
\end{enumerate}
\end{thm}

\subsection{}Finally, we complete the classification of \( \GL_2(F) \)-distinguished principal series representations of \( \GL_2(E) \) in the case \( \ell = 2 \), which was left open in \cite{cui2024}.  

\begin{thm}[Proposition \ref{pil=2} and Theorem \ref{spl=2}]\label{thm:GL2}
Assume \( \ell = 2 \).
\begin{enumerate}
\item A principal series representation \( \pi(\chi_1, \chi_2) \) of \( \GL_2(E) \) is \( \GL_2(F) \)-distinguished if and only if  
\(
\chi_1 \chi_2^\sigma = \boldsymbol{1}
\)
or  
\(
\chi_1 |_{F^\times} = \chi_2 |_{F^\times} = \boldsymbol{1} \ \text{with} \ \chi_1 \neq \chi_2.
\)
\item The special representation \( \Sp_\chi \) is \( \GL_2(F) \)-distinguished if and only if \( \chi|_{F^\times} = \boldsymbol{1} \).
\end{enumerate}
\end{thm}
\begin{rem}
The comparison in Theorem \ref{thm:equiv} is specific to supercuspidal representations and does not extend to all cuspidal representations. For example, suppose that \(\ell=2\) and that \(E/F\) is unramified. The special representation \(\mathrm{Sp}_\chi\), $\chi|_{F^\times}=\boldsymbol{1}$ is cuspidal but not supercuspidal, and Theorems \ref{thm:principal} and \ref{thm:GL2} give
\[
\operatorname{Hom}_{D^\times}(\mathrm{Sp}_\chi,\mathbf1)=0,
\qquad
\operatorname{Hom}_{\mathrm{GL}_2(F)}(\mathrm{Sp}_\chi,\mathbf1)\ne0.
\]
\end{rem}

The paper is organized as follows. Section 2 sets up the notation. In Section 3 we collect the cohomological preliminaries needed for the Ext-group computations. Section 4 is devoted to the classification of \( D^\times \)-distinguished principal series representations of $\GL_2(E)$ and their subquotients, which is the most technique part of the whole paper, involving a lot of explicit computation for the group homology. In Section 5 we treat the supercuspidal case and prove the reduction to \( \mathrm{GL}_2(F) \)-distinction. Finally, Section 6 completes the classification of \( \mathrm{GL}_2(F) \)-distinguished principal series when \( \ell = 2 \).

\subsection*{Acknowlegement} The second named author  was partially supported by NSFC 12301031.

    \section{Notation}
    Let $F$ be a locally compact non-Archimedean local field of characteristic zero and residual characteristic $p$. The residue field of $F$ is $\mathbb{F}_q$, where $q$ is the power of $p$. Denote by $D$ the unique 4-dimensional quaternion division algebra defined over F. The elemets in $D$ are in the form of $x+yi+zj+wk$ with $i^2=a\in F^\times,j^2=b\in F^\times,ij=-ji=k$ and $x,y,z,w\in F$. Let $E=F[i]$, then $E/F$ is a quadratic field extension. Denote by $q_F=q$ (resp. $q_E$) the cardinality of the residue field of $F$ (resp. $E$). Let $\sigma$ be the non-trivial $F$-automorphism of $E$.

    Denote by $\omega_{E/F}$ the quadratic character of $F^\times$ associated to the quadratic ﬁeld extension $E/F$ by the local class ﬁeld theory. Let Nm$_{E/F}$ (resp. Nrd$_{D/F}$) be the norm (resp. reduced norm) map of $E/F$ (resp. $D/F$).

    We denote by GL$_2$ the general linear group ,$B_2$ the standard Borel sungroup of GL$_2$.There exists an embedding
    \begin{align*}
        D^\times=\left\{\begin{pmatrix}
            \alpha&b\beta\\
            \sigma(\beta)&\sigma(\alpha)
        \end{pmatrix}|\alpha,\beta\in E,\Nm_{E/F}(\alpha)-b\Nm_{E/F}(\beta)\neq0\right\}\hookrightarrow \GL_2(E)
    \end{align*}
    For a character $\chi$ of $E^\times$,we will also denote by $\chi$ the character $\chi\circ \mathrm{det}$ of $\GL_2(E)$.The trivial charater will be denoted by $\boldsymbol{1}$. Let $\nu_2$ be the character $g\mapsto|\det(g)|$ of $\GL_2(E)$.

Let $G$ be a locally profinite group and let $A$ be a $G$-module equipped with a continuous left action of $G$ by automorphisms, denoted by $g \cdot a$ for $g \in G$ and $a \in A$. The first cohomology group is defined as
\[
\mathrm{H}^1(G,A) = \frac{
  \left\{ f: G \to A \;\middle|\; f(gh) = f(g) + g \cdot f(h),\ \forall g,h \in G, \mathrm{and}\ f\ \mathrm{is\ continuous} \right\}
}{
  \left\{ f:G\to A \;\middle|\; \exists a \in A \mbox{ such that } \ f(g) = g \cdot a - a,\ \forall g \in G \right\}
}
\] If $A$ is a trivial $G$-module, then
\[
\mathrm{H}^1(G,A) = \Hom_{cts}(G,A)
\]
where $\Hom_{cts}(G,A)$ denotes the group of continuous homomorphisms from $G$ to the additive group $A$.

For a closed normal subgroup $N$ in $G$, $G$ acts on $\mathrm{H}^1(N,A)$ by
\[
(g \cdot f)(n) = g \cdot f(g^{-1} n g), \quad n \in N,\ g \in G
\]
for a cocycle $f: N \to A$, this action factors through the quotient $G/N$. Moreover, 
define $$A^N := \{ a \in A \mid n\cdot a = a,\ \forall n \in N \}.$$
By \cite[Theorem III.2]{HochschildSerre1953}, the corresponding five-term exact sequence arising from the Hochschild--Serre spectral sequence is:
\[
0 \to \mathrm{H}^1(G/N, A^N) \xrightarrow{\mathrm{inf}} \mathrm{H}^1(G, A) \xrightarrow{\mathrm{res}} \mathrm{H}^1(N, A)^{G/N} \xrightarrow{d} \mathrm{H}^2(G/N, A^N) \xrightarrow{\mathrm{inf}} \mathrm{H}^2(G, A),
\]
where $\mathrm{inf}$ is the inflation map, $\mathrm{res}$ is the restriction map, and $d$ is the transgression map.

    Let $\pi$ be an irreducible representation of $\GL_2(E)$. We say that $\pi$ is $D^\times$-distinguished or simply distinguished if $\Hom_{D^\times}(\pi,\boldsymbol{1})\neq 0$.

    \section{Cohomological Preliminaries}
    In this section, we collect several cohomological results that will be used in the proofs of the distinction criteria in the following sections. We always assume $\ell\neq p$.

\begin{prop}\label{coE}
Let \(E/F\) be a quadratic extension of non-archimedean local fields of residue 
characteristic \(p\), and let \(\ell \neq p\) be a prime. Let \(q_E\) denote the 
cardinality of the residue field of \(E\). Then
\[
\dim_{\overline{\boldsymbol F_\ell}} \mathrm{H}^1(E^\times,\boldsymbol{1})
=
\begin{cases}
2, & \ell \mid q_E - 1, \\[4pt]
1, & \ell \nmid q_E - 1.
\end{cases}
\]

\end{prop}

\begin{proof}
Note that $\mathrm{H}^1(E^\times,\boldsymbol{1})=\Hom(E^\times,\overline{\boldsymbol{F}_\ell})$ and there is a standard decomposition
\[
E^\times \simeq \langle \pi_E \rangle \times \mathcal O_E^\times,
\qquad
\mathcal O_E^\times \simeq \boldsymbol{k}_E^\times \times (1 + \mathfrak p_E),
\]
where \(\pi_E\) is a uniformizer of \(E\), and \(\boldsymbol{k}_E \) 
is its residue field. Since \(\ell \neq p\), the pro-\(p\) group 
\(1 + \mathfrak p_E\) has no non-trivial homomorphism to \(\overline{\boldsymbol F_\ell}\). 
Hence
\[
\dim\ \Hom(E^\times, \overline{\boldsymbol F_\ell})
=
\dim\ \Hom(\langle \pi_E \rangle, \overline{\boldsymbol F_\ell})
+
\dim\ \Hom(\boldsymbol{k}_E^\times, \overline{\boldsymbol F_\ell}).
\]

The first factor is one-dimensional since \(\langle \pi_E \rangle \simeq \mathbb Z\). 
For the second factor, \(\boldsymbol{k}_E^\times\) is cyclic of order \(q_E - 1\), so its 
homomorphisms to \(\overline{\boldsymbol F_\ell}\) are non-zero exactly when 
\(\ell \mid q_E - 1\), in which case the space is one-dimensional. 
This finishes the proof.
\end{proof}

\begin{coro}
\label{coEF}
With the same notations and assumptions as in Proposition \ref{coE}, we have

\[
\dim_{\overline{\boldsymbol F_\ell}} \mathrm{H}^1(E^\times/F^\times, \mathbf{1})
=
\begin{cases}
1, &\mbox{if }E/F \text{ is unramified} \text{ and } \ell \mid q + 1 , \\[4pt]
1, &\mbox{if } E/F \text{ is ramified} \text{ and } \ell = 2 , \\[4pt]
0, & \text{otherwise}.
\end{cases}
\]
where $q$ denote the cardinality of the residue field of $F$.
\end{coro}
\begin{proof}
We have \(\mathrm{H}^1(E^\times/F^\times, \mathbf{1}) = \operatorname{Hom}(E^\times/F^\times, \overline{\boldsymbol F_\ell})\).
If $E/F$ is an unramified quadratic extension, $E^\times/F^\times$ is the direct product of the cyclic group $C_{q+1}$ of order $q+1$ and a pro-$p$ group. Then $$\Hom(E^\times/F^\times,\overline{\boldsymbol{F}_\ell})=\Hom(C_{q+1},\overline{\boldsymbol{F}_\ell})$$ since $\ell\neq p$ and the dimension is 1 if and only if $\ell|q+1$.

If $E/F$ is a ramified quadratic extension, when $p\neq2$, $E^\times/F^\times$ is the direct product of the cyclic group $C_2$ of order 2 and a pro-$p$ group. Then $\Hom(E^\times/F^\times,\overline{\boldsymbol{F}_\ell})=\Hom(C_{2},\overline{\boldsymbol{F}_\ell})$ and the dimension is 1 if and only if $\ell=2$. When $p=2$, $E^\times/F^\times$ is a pro-$2$ group so $\Hom(E^\times/F^\times, \overline{\boldsymbol{F}_\ell})=0$ since $\ell\neq 2$.

This proves the formula.
\end{proof}

\begin{lem}\label{l51}
   Let $\chi_1,\chi_2$ be the characters of $F^\times$. Then $\Ext^1_{F^\times\times F^\times}(\boldsymbol{1},\chi_1\boxtimes\chi_2)\neq0$ if and only if $\chi_1|_{F^\times}=\chi_2|_{F^\times}=\boldsymbol{1}$.
\end{lem}
\begin{proof}
    By the well-known K$\ddot{\mathrm{u}}$nneth formula, we obtain that
    \begin{align*}
        \Ext^1_{F^\times\times F^\times}(\boldsymbol{1},\chi_1\boxtimes\chi_2)=\mathrm{H}^1(F^\times\times F^\times,\chi_1\boxtimes\chi_2)\cong(\mathrm{H}^1(F^\times,\chi_1)\otimes \mathrm{H}^0(F^\times,\chi_2))\oplus(\mathrm{H}^1(F^\times,\chi_2)\otimes \mathrm{H}^0(F^\times,\chi_1))
    \end{align*}
where $\mathrm{H}^0(F^\times,\chi):=\{v|\chi(a)v=v,\forall a\in F^\times\}$. Therefore, $\mathrm{dim\ H}^0(F^\times,\chi)=\begin{cases}
    1,\ \chi|_{F^\times}=\boldsymbol{1};
    \\0,\ \chi|_{F^\times}\neq\boldsymbol{1}.\end{cases}$

If $\chi|_{F^\times}=\boldsymbol{1}, \mathrm{H}^1(F^\times,\chi)=\Hom(F^\times,\overline{\boldsymbol{F}_\ell})$. By the same argument in Proposition \ref{coE},
\begin{align*}
    \dim\ \Hom(F^\times,\overline{\boldsymbol{F}_\ell})=\begin{cases}
        1,\ \ell\nmid q-1;\\
        2,\ \ell\ |\ q-1.
    \end{cases}
\end{align*}
If $\chi|_{F^\times}\neq\boldsymbol{1}$, for any cocycle $f\in\mathrm{H}^1(F^\times,\chi)$, taking $a\in F^\times$ such that $\chi(a)\neq 1$, we have $\forall b\in F^\times$,
\begin{align*}
    f(b)=\dfrac{\chi(b)-1}{\chi(a)-1}f(a)
\end{align*}
since $f(ab)=f(ba)$. Therefore $f$ is a coboundry i.e. $\mathrm{H}^1(F^\times,\chi)=0$.

In summary, $\Ext^1_{F^\times\times F^\times}(\boldsymbol{1},\chi_1\boxtimes\chi_2)\neq0$ if and only if $\chi_1|_{F^\times}=\chi_2|_{F^\times}=\boldsymbol{1}$. Moreover, in this case,
\begin{align*}
    \mathrm{dim}\ \Ext^1_{F^\times\times F^\times}(\boldsymbol{1},\chi_1\boxtimes\chi_2)=\begin{cases}
    2,\ \ell\nmid q-1;\\
        4,\ \ell\ |\ q-1.\\
    \end{cases}
\end{align*}
This finishes the proof.
\end{proof}

\begin{prop}\label{cofinite}
    Suppose $G$ is a finite $p$-group, i.e., its order $|G|$ is a power of $p$, $A$ is a discrete $\ell$-torsion $G$-module and $\ell\neq p$. Then $\mathrm{H}^r(G,A)=0$ for $r\geq 1$.
\end{prop}
\begin{proof}
    Followed by \cite[Corollary 1.31]{milneCFT}, if $|G|=p^n$, we have $p^n\mathrm{H}^r(G,A)=0$ for $r\geq1$. 
    
    We claim that $\mathrm{H}^r(G,A)$ is an $\ell$-torsion, i.e., annihilated by $\ell$. Consider an arbitrary function $\varphi:G^r\longrightarrow A\in \mathrm{H}^r(G,A)$. Since $G$ is finite, $\varphi(G)$ is finite. Denote $\varphi(G)=\{a_1,...,a_k\}$. Since $A$ is an $\ell$-torsion, $\ell\varphi(G)=0$ i.e. $\ell\varphi=0$.

    Since $\ell$ and $p^n$ are coprime, there exists integers $u,v$ such that $u\ell+vp^n=1$. For any $\varphi\in\mathrm{H^r(G,A)},$\begin{align*}
        \varphi=(u\ell+vp^n)\varphi=u\ell\varphi+vp^n\varphi=0
    \end{align*}
    This finishes the proof.
\end{proof}

\begin{coro}\label{copro-p}
    Suppose $G$ is a pro-$p$ group, $A$ is a discrete $\ell$-torsion $G$-module and $\ell\neq p$. Then $\mathrm{H}^r(G,A)=0$ for $r\geq 1$.
\end{coro}
\begin{proof}
    According to \cite[Proposition 4.2]{milneCFT}, we have $\mathrm{H}^r(G,A)=\lim\limits_{\longrightarrow}\mathrm{H}^r(G/H,A^H)$ as $H$ runs through the open normal subgroups of $G$. Since $G$ is a pro-$p$ group, $G/H$ is a finite $p$-group and $A^H\subset A$ is also a discrete $\ell$-torsion $G/H$-module. Using Proposition \ref{cofinite}, $\mathrm{H}^r(G/H,A^H)=0$. Then $\mathrm{H}^r(G,A)=0$ for $r\geq 1$.
\end{proof}

\begin{prop}\label{extfix}
    Let $G$ be a profinite group and $U$ be the normal pro-$p$ subgroup of $G$. Assume $\ell\neq p$. Suppose $V$ is a smooth $\overline{\boldsymbol{F}_\ell}$-representation of $G$ and $\boldsymbol{1}$ is the trivial representation of $G$. Taking $U$-invariant induces an isomorphism
    \begin{align*}
        \Ext_G^1(V,\boldsymbol{1})\cong\Ext_{G/U}^1(V^U,\boldsymbol{1})
    \end{align*}
    where $\Ext$ groups are computed in the respective categories of smooth representations.
\end{prop}
\begin{proof}
    Let $M$ be an extension of $V$ by the trivial character $\boldsymbol{1}$ in the categories of smooth representations of $G$. There exists an exact sequence:
    \begin{align*}
        0\longrightarrow\boldsymbol{1}\longrightarrow M \stackrel{f}{\longrightarrow} V\longrightarrow0
    \end{align*}
    where $f:M\rightarrow V$ is the quotient map. 
    Since $\ell\neq p$ and $U$ is a pro-$p$ group, the $U$-invariant functor $(-)^U$ is exact due to Corollary \ref{copro-p}. Then we have the following exact sequence
    \begin{align*}
        0\longrightarrow\boldsymbol{1}\longrightarrow M^U \longrightarrow V^U\longrightarrow0.
    \end{align*}
    Then we define a map
    \begin{align*}
\Phi:\Ext_G^1(V,\boldsymbol{1})&\longrightarrow\Ext_{G/U}^1(V^U,\boldsymbol{1})\\
M&\longmapsto M^U.
    \end{align*}
    
    Conversely, if $M'$ is an extension of $V^U$ by the trivial character $\boldsymbol{1}$ in the categories of smooth representations of $G/U$, then
    \begin{align*}
        0\longrightarrow\boldsymbol{1}\longrightarrow M' \stackrel{f'}{\longrightarrow} V^U\longrightarrow0
    \end{align*}
    is exact.
    Followed by \cite[Section 12.6:Extensions]{stacks-project}, denote $e_U:V\longrightarrow V^U$ the averaging projection. Then $M'$ has a pullback $\overline{M'}=M'\times_{V^U}V$, which is defined as:
    \begin{align*}
        \overline{M'}:=\{(x,v)|x\in M',v\in V\ \mathrm{satisfying}\ f'(x)=e_U(v)\}.
    \end{align*}
    Then we define another map
    \begin{align*}
        \Psi:\Ext_{G/U}^1(V^U,\boldsymbol{1})&\longrightarrow\Ext_G^1(V,\boldsymbol{1})\\
        M'&\longmapsto\overline{M'}
    \end{align*}

    Next we need to prove that $\Psi,\Phi$ are mutually inverse. Consider $\Psi\circ\Phi$, $\forall x\in M, \Psi\circ\Phi(x)=(e_U(x),f(x))$. Then $\Psi\circ\Phi$ induces isomophisms $\boldsymbol{1}\cong\boldsymbol{1}$ and $V\cong\Psi\circ\Phi(V)$ in the category of smooth representaions of $G$. By the snake Lemma, $M\cong\Psi\circ\Phi(M)$.

    For $\Phi\circ\Psi$, $\forall x'\in M'$,
    \begin{align*}
        \Phi\circ\Psi(x)=\{(x,v)|x\in M',v\in V^U\ \mathrm{satisfying}\ f'(x)=v\}.
    \end{align*}
    Thus we have an isomorphism $M'\cong\Phi\circ\Psi(M')$. This finishes the proof.
\end{proof}

\begin{coro}\label{exth1}
    With the same notations as in Proposition \ref{extfix}, suppose $\ell\neq p$ and $G/U$ is a finite group. Then there exists an isomorphism
    \begin{align*}
        \Ext_G^1(V,\boldsymbol{1})\cong \mathrm{H}^1(G/U,(V^U)^\vee)
    \end{align*}
    where $(V^U)^\vee$ is the dual space of $V^U$.
\end{coro}
\begin{proof}
    Followed by Proposition \ref{extfix} and \cite[Theorem 3.2.1]{Jenda00},
    \begin{align*}
    \Ext^1_G(V,\boldsymbol{1})\cong\Ext_{G/U}^1(V^U,\boldsymbol{1})\cong\Ext^1_{G/U}(\boldsymbol{1},(V^U)^\vee)
    \end{align*} 
    then $\Ext^1_{G/U}(\boldsymbol{1},(V^U)^\vee)\cong\mathrm{H}^1(G/U,(V^U)^\vee)$.
\end{proof}

Now, let $\chi$ be a character of $F^\times$. We will consider $\Ext^1_{\GL_2(F)}(\boldsymbol{1},\chi\circ\det)$ and $\Ext^1_{D^\times}(\boldsymbol{1},\chi\circ\det)$. For the former, the cases when $\chi\neq\boldsymbol{1}$ and when $\ell\neq2$ with $\chi=\boldsymbol{1}$ have already been discussed in \cite[Lemma 4.4]{cui2024}. We will complete the case when $\ell=2$ with $\chi=\boldsymbol{1}$.

\begin{lem}
    Assume $\ell=2$. Let $\boldsymbol{1}$ be the trivial character of $F^\times$. Then $\mathrm{dim}_{\overline{\boldsymbol{F}_2}}\Ext^1_{\GL_2(F)}(\boldsymbol{1},\boldsymbol{1})=2.$
\end{lem}
\begin{proof}
    There exists a short sequence:
    \begin{align*}
        1\longrightarrow\SL_2(F)\longrightarrow\GL_2(F)\stackrel{\det}{\longrightarrow}F^\times\longrightarrow1
    \end{align*}
    By the Hochschild-Serre spectral sequence, we have 
    \begin{align*}
        0\longrightarrow\mathrm{H}^1(F^\times,\boldsymbol{1})\longrightarrow\mathrm{H}^1(\GL_2(F),\boldsymbol{1})\longrightarrow\mathrm{H}^1(\SL_2(F),\boldsymbol{1})^{F^\times}.
    \end{align*}
    For $\mathrm{H}^1(\SL_2(F),\boldsymbol{1})$, it equals $\Hom(\SL_2(F),\overline{\boldsymbol{F}_2})$. Since $\overline{\boldsymbol{F}_2}$ is Abelian, we consider $\SL_2(F)^{ab}$, the Abelianization of $\SL_2(F)$ which is the identity element since $\SL_2(F)$ is perfect. Then, $\mathrm{H}^1(\SL_2(F),\boldsymbol{1})=0$.

    Followed by the spectral sequence above and the same argument in Proposition \ref{coE}:
    \begin{align*}
        \Ext^1_{\GL_2(F)}(\boldsymbol{1},\boldsymbol{1})=\mathrm{H}^1(\GL_2(F),\boldsymbol{1})=\mathrm{H}^1(F^\times,\boldsymbol{1})=\Hom(F^\times,\overline{\boldsymbol{F}_2})
    \end{align*}
    and its dimension is 2 since $\ell\ |\ (q-1)$.
\end{proof}
    The rest of this section will focus on the proof of the following key lemma, which will be used in the next section. 
    \begin{lem}[Key lemma]\label{3.3}
         Assume $\ell\neq2$ and $p\neq2$. Let $\chi$ be a character of $F^\times$ and denote $\chi\circ\det$ the character of $D^\times
         $. Let $Z=F^\times$ be the center of $D^\times$.
         \begin{enumerate}
             \item If $\chi^2\neq\boldsymbol{1}$, then $\mathrm{dim}$\ $\Ext_{D^\times}^1(\boldsymbol{1},\chi\circ\det)=0$.
             \item If $\chi^2=\boldsymbol{1}$ and $\chi\neq\boldsymbol{1}$, then
             \begin{align*}
                 \mathrm{dim}\ \Ext_{D^\times}^1(\boldsymbol{1},\chi\circ\det)=\begin{cases}
            1,\ \mathrm{if}\ \chi\circ\det|_{C_{q+1}}=\boldsymbol{1}\ \mathrm{and}\ \ell|q+1;\\
            0,\ \mathrm{others,}
        \end{cases}
             \end{align*}
             where $C_{q+1}=\boldsymbol{k}_D^\times/\boldsymbol{k}_F^\times.$
             Moreover, $\mathrm{dim}\ \Ext_{D^\times}^1(\boldsymbol{1},\chi\circ\det)=\mathrm{dim}\ \Ext_{PD^\times}^1(\boldsymbol{1},\chi\circ\det)$.
             \item If $\chi=\boldsymbol{1}$, $\mathrm{dim}$\  $\Ext_{PD^\times}^1(\boldsymbol{1},\boldsymbol{1})=0$.
         \end{enumerate}
         Here $\Ext_{D^\times}^1$(resp. $\Ext_{PD^\times}^1$) is the $\Ext$ functor in the category of the smooth representation of $D^\times$(resp. $PD^\times$). 
     \end{lem}
     \begin{proof}
     \begin{enumerate}
         \item 
         Note that $\Ext_{D^\times}^1(\boldsymbol{1},\chi\circ\det)=\mathrm{H}^1(D^\times,\chi\circ\det)$. There exists a short sequence:
         \begin{align*}
             1\longrightarrow Z\longrightarrow D^\times\longrightarrow PD^\times\longrightarrow 1
         \end{align*}
         By the Hochschild-Serre spectral sequence, we have
         \begin{align*}\label{f1}
             0\longrightarrow\mathrm{H}^1(PD^\times,(\chi\circ\det)^{Z})\longrightarrow\mathrm{H}^1(D^\times,\chi\circ\det)\longrightarrow\mathrm{H}^1(Z,(\chi\circ\det))^{PD^\times}.\tag{3-1}
         \end{align*}
         Consider $(\chi\circ\det)^{Z}$. For any $v\in (\chi\circ\det)^{Z}$ and $z\in Z$, we have $$z\cdot v=\chi(\det(z))v=\chi^2(z)v=v.$$
         Hence $(\chi\circ\det)^{Z}$ is vanishing if and only if $\chi^2\neq\boldsymbol{1}$.\\

         When $\chi^2\neq\boldsymbol{1}$, $\mathrm{H}^1(PD^\times,(\chi\circ\det)^{Z})=0$. For $\mathrm{H}^1(Z,\chi\circ\det)$, given $\varphi\in\mathrm{H}^1(Z,\chi\circ\det)$, $\varphi$ satisfies $\varphi(z_1z_2)=\chi^2(z_1)\varphi(z_2)+\varphi(z_1),z_1,z_2\in F^\times.$ Because $\varphi(z_1z_2)=\varphi(z_2z_1),$ we have $$\chi^2(z_1)\varphi(z_2)+\varphi(z_1)=\chi^2(z_2)\varphi(z_1)+\varphi(z_2).$$ Then $\chi^2\neq \boldsymbol{1}$ implies the existence of $z_0$ which satisfies $\chi^2(z_0)\neq1$. Fix $z_1=z_0$. Therefore, $\forall z_2\in F^\times$, $\varphi(z_2)=(\chi^2(z_2)-1)\dfrac{\varphi(z_0)}{\chi^2(z_0)-1}$, and so $\varphi$ is a coboundry. Then $\mathrm{H}^1(Z,(\chi\circ\det))=0$ and so $\mathrm{H}^1(D^\times,\chi\circ\det)=0$.\\
\item 
         When $\chi^2=\boldsymbol{1}$ but $\chi\neq\boldsymbol{1}$, $\mathrm{H}^1(PD^\times,(\chi\circ\det)^{Z})=\mathrm{H}^1(PD^\times,\chi\circ\det).$ Let $O_D$ be the valuation ring of $D$, $\mathfrak{p}_D$ be its unique maximal ideal and $\boldsymbol{k}_D\cong \mathbb{F}_{q^2}$ be its residual field. We define
         \begin{align*}
             U:=O_D^\times Z/Z\cong O_D^\times/O_F^\times,\quad  U_1:=(1+\mathfrak{p}_D)Z/Z\cong1+\mathfrak{p}_D/1+\mathfrak{p}_F.
         \end{align*}
         Then $U/U_1=\boldsymbol{k}_D^\times/\boldsymbol{k}_F^\times\cong C_{q+1}$ which is a cyclic group of order $q+1$. There exists a short sequence:
         \begin{align*}
             1\longrightarrow U\longrightarrow PD^\times \stackrel{\mu}{\longrightarrow}\mathbb{Z}/2\mathbb{Z}\longrightarrow1
         \end{align*}
         where $\mu$ is the exponential valuation on $D^\times$. Therefore, for $PD^\times/U_1$, we have 
         \begin{align*}
             1\longrightarrow U/U_1\longrightarrow PD^\times/U_1 \stackrel{\bar{\mu}}{\longrightarrow}\mathbb{Z}/2\mathbb{Z}\longrightarrow1.
         \end{align*}
         We claim that the group $PD^\times/U_1$ is isomorphic to the dihedral group $D_{2(q+1)}=C_{q+1}\rtimes C_2$. Actually we just need to prove the existence of the section $s:\mathbb{Z}/2\mathbb{Z}\longrightarrow PD^\times/U_1$ such that $\bar{\mu}\circ s=\mathrm{id}.$ We define 
         \begin{align*}
             s:\mathbb{Z}/2\mathbb{Z}\longrightarrow PD^\times/U_1,\quad s(0)=1,s(1)=\bar{\Pi}
         \end{align*}
         where $\Pi$ is the prime element of $O_D$ and $\bar{\Pi}$ is the image of $\Pi Z$ in $PD^\times/U_1$. Then $s$ is a homomorphism and $\bar{\mu}\circ s=\mathrm{id}$.

         Now we have the exact sequence of groups:
        \begin{align*}
            1\longrightarrow U_1\longrightarrow PD^\times\longrightarrow D_{2(q+1)}\longrightarrow1.
        \end{align*}
        By the Hochschild-Serre spectral sequence, we have
        \begin{align*}\label{f2}
            0\longrightarrow \mathrm{H}^1(D_{2(q+1)},(\chi\circ\det)^{U_1})\longrightarrow\mathrm{H}^1(PD^\times,\chi\circ\det)\longrightarrow \mathrm{H}^1(U_1,\chi\circ\det)^{D_{2(q+1)}}.\tag{3-2}
        \end{align*}
        Since $U_1$ is a pro-$p$ group, $\chi$ is a quadratic character and $p\neq 2$, $\chi\circ\det$ is trivial on $U_1$ and $$\mathrm{H}^1(U_1,\chi\circ\det)=\Hom(U_1,\overline{\boldsymbol{F}_\ell})=0.$$ Then 
        $\mathrm{H}^1(PD^\times,\chi\circ\det)= \mathrm{H}^1(D_{2(q+1)},(\chi\circ\det)^{U_1})= \mathrm{H}^1(D_{2(q+1)},\chi\circ\det)$. By the Hochschild-Serre spectral sequence, one has the following exact sequence
        \begin{align*}\label{f3}
            0\longrightarrow\mathrm{H}^1(C_2,(\chi\circ\det)^{C_{q+1}})\longrightarrow \mathrm{H}^1(D_{2(q+1)},\chi\circ\det)\\\longrightarrow\mathrm{H}^1(C_{q+1},\chi\circ\det)^{C_2}\longrightarrow\mathrm{H}^2(C_2,(\chi\circ\det)^{C_{q+1}})\tag{3-3}
        \end{align*}
        
        If $\chi\circ\det|_{C_{q+1}}\neq\boldsymbol{1}$, given any $ a \in (\chi\circ\det)^{C_{q+1}}$ and $r\in C_{q+1}$, we have $$r\cdot a=\chi\circ\det(r)a=a,$$ which implies that $a$ must be zero. Then $\mathrm{H}^1(C_2,(\chi\circ\det)^{C_{q+1}})$=0. Since $C_{q+1}$ is a cyclic group and $\chi\circ\det|_{C_{q+1}}\neq\boldsymbol{1}$, $\mathrm{H}^1(C_{q+1},\chi\circ\det)$=0. In this case, we have $\mathrm{H}^1(D_{2(q+1)},\chi\circ\det)$=0.

        If $\chi\circ\det|_{C_{q+1}}=\boldsymbol{1}$, 
        then $\chi\circ\det|_{C_2}\neq\boldsymbol{1}$. Otherwise $\chi$ is trivial,  contradicting the condition that $\chi\neq\mathbf{1}$. In this case, 
        \begin{equation}\label{first:van}\mathrm{H}^1(C_2,(\chi\circ\det)^{C_{q+1}})=\mathrm{H}^1(C_2,\chi\circ\det)=0\tag{3-4}\end{equation} since $C_2$ is also a cyclic group and $\chi\circ\det|_{C_2}\neq\boldsymbol{1}$. It remains to determine $\mathrm{H}^1(C_{q+1}, \chi \circ \det)^{C_2}.$
Since \(C_{q+1}\) acts trivially on the coefficient module \(\overline{\boldsymbol{F}_\ell}\), we have that
$$
\mathrm{H}^1(C_{q+1}, \chi \circ \det) \cong \Hom(C_{q+1}, \overline{\boldsymbol{F}_\ell}).$$

For \(f \in \operatorname{Hom}(C_{q+1}, \overline{\boldsymbol{F}_\ell})\) and \(s \in C_2\) of order $2$, the action of \(C_2\) is given by
\[
(s \cdot f)(r) = (\chi \circ \det)(s)f(s^{-1} r s)=(\chi \circ \det)(s)f(r^{-1})= -(\chi \circ \det)(s) f(r), \quad \forall r \in C_{q+1}.
\]
Because \(\chi \circ \det|_{C_2} \neq \boldsymbol{1}\), we have$(\chi \circ \det)(s) = -1.$
Therefore, $(s \cdot f)(r) = f(r)$
and so every \(f \in \Hom(C_{q+1}, \overline{\boldsymbol{F}_\ell})\) is invariant under \(C_2\). Hence
$\operatorname{Hom}(C_{q+1}, \overline{\boldsymbol{F}_\ell})^{C_2}
= \operatorname{Hom}(C_{q+1}, \overline{\boldsymbol{F}_\ell}).$
Now we distinguish two cases: $\ell\nmid\ q+1$ and $\ell|\ q+1$.
\begin{itemize}[leftmargin=0pt]
\item If \(\ell \nmid q+1\), then for any \(f \in \operatorname{Hom}(C_{q+1}, \overline{\boldsymbol{F}_\ell})\), we have
$(q+1) f(r) = f(r^{q+1}) = f(1) = 0$
which implies \(f(r) = 0\). Thus
$\operatorname{Hom}(C_{q+1}, \overline{\boldsymbol{F}_\ell}) = 0$
and consequently
\[
H^1(D_{2(q+1)}, \chi \circ \det) = 0.
\]

\item If \(\ell \mid q+1\), then \(\operatorname{Hom}(C_{q+1}, \overline{\boldsymbol{F}_\ell})\) is a one-dimensional \(\overline{\boldsymbol{F}_\ell}\)-vector space.
Moreover, in the case when \(\ell \mid q+1\) and \(\chi \circ \det|_{C_2} \neq \boldsymbol{1}\), 
we have \(M = \overline{\boldsymbol F_\ell}\) with \(C_2\) acting via
$$s \cdot m = (\chi \circ \det)(s) \, m = -m, \forall m \in M.$$
Since \(\ell \neq 2\), we have
\[
M^{C_2} = \{ m \in M \mid m = -m \} = 0.
\]
For the cyclic group \(C_2\), the standard cohomology formula gives the following isomorphism
\[
\mathrm{H}^2(C_2, M) \cong \mathrm{H}_T^0(C_2,M)= M^{C_2} / (1 + s)M.
\]
where $\mathrm{H_T^0(C_2,M)}$ denotes the Tate cohommology group (see \cite[Section 2]{milneCFT}).  Since \(M^{C_2} = 0\), it follows immediately that
$\mathrm{H}^2(C_2, M) = 0.$
Thus, by the spectral sequence (\ref{f3})
the last term $\mathrm{H}^2(C_2,(\chi\circ\det)^{C_{q+1}})$ vanishes. Together with \eqref{first:van}, it forces that
\[
\mathrm{dim}\ \mathrm{H}^1(D_{2(q+1)}, \chi\circ\det)=\mathrm{dim}\ \mathrm{H}^1(C_{q+1},\chi\circ\det)^{C_2}=\mathrm{dim}\ \Hom(C_{q+1},\overline{\boldsymbol{F}_\ell}) = 1.
\]
\end{itemize}

        For $\mathrm{H}^1(Z,(\chi\circ\det))^{PD^\times}$, it equals $\Hom(Z,\overline{\boldsymbol{F}_\ell})^{PD^\times}$ due to the fact that $\chi^2=\boldsymbol{1}$. Then given any $f
        \in \Hom(Z,\overline{\boldsymbol{F}_\ell})$ and $a\in PD^\times$, $$a\cdot f(z)=\chi\circ\det(a)f(z)=f(z). $$ Since $\chi\circ\det$ is not trivial, $f=0$ which means $\mathrm{H}^1(Z,(\chi\circ\det))^{PD^\times}=0$ . 
        
    In a short summary, when $\chi^2=\boldsymbol{1}$ but $\chi\neq\boldsymbol{1}$,
    \begin{align*}
        \mathrm{dim}\ \Ext_{D^\times}^1(\boldsymbol{1},\chi\circ\det)=\begin{cases}
            1,\ \mathrm{if}\ \chi\circ\det|_{C_{q+1}}=\boldsymbol{1}\ \mathrm{and}\ \ell|q+1;\\
            0,\ \mathrm{others.}
        \end{cases}
    \end{align*}\\
\item 
        When $\chi=\boldsymbol{1}$, we have $\Ext_{PD^\times}^1(\boldsymbol{1},\boldsymbol{1})=\mathrm{H}^1(PD^\times,\boldsymbol{1})$. By the exact sequence (\ref{f2}), $\mathrm{H}^1(U_1,\boldsymbol{1})=0$. Then
        \begin{align*}
\mathrm{H}^1(PD^\times,\boldsymbol{1})=\mathrm{H}^1(D_{2(q+1)},\boldsymbol{1})=\Hom(D_{2(q+1)},\overline{\boldsymbol{F}_\ell})=\Hom(D_{2(q+1)}^{ab},\overline{\boldsymbol{F}_\ell})
        \end{align*}
    where $D_{2(q+1)}^{ab}\cong C_2\times C_2$ is the  Abelianization of $D_{2(q+1)}$. Since $\ell\neq 2$, it is zero.
     \end{enumerate}This finishes the proof.
    \end{proof}

    \section{The Principal Series Representations}
    This section is devoted to the study of $D^\times$-distinguished  $\ell$-modudar principal series representations of GL$_2(E)$. Let $G=\GL_2(E)$ in this section.
    
    We first give some lemmas that do not depend on the choices of \(\ell\) and \(p\).

    \begin{lem}\label{3.1}
    The number of double cosets of $\GL_2(E)$ with respect to $D^\times$ and $B_2(E)$ is $\mathrm{1}$, which means $\GL_2(E)=B_2(E)D^\times$.
    \end{lem}
    \begin{proof}
        We have $\GL_2(E)/B_2(E)\simeq\mathbb{P}^1(E)$, so we only need to consider the number of orbits of the action of $D^\times$ on $\mathbb{P}^1(E)$.The action of $D^\times$ on $\mathbb{P}^1(E)$ is defined as
        \begin{align*}
            \begin{pmatrix}
                \alpha&b\beta\\
            \sigma(\beta)&\sigma(\alpha)
            \end{pmatrix}\cdot x:=\dfrac{\alpha x+b\beta}{\sigma(\beta)x+\sigma(\alpha)}
        \end{align*}
     We claim that the action is transitive, that is, $\forall x\in\mathbb{P}^1(E)$,there exists $d\in D^\times$ such that $d\cdot0=x$.

     If $x=\infty$,let $d=\begin{pmatrix}
         0&b\\
         1&0
     \end{pmatrix}\in D^\times$,then $d\cdot0=\dfrac{b}{0}=\infty$.

     If $x=t\in E$,let $d=\begin{pmatrix}
         1&t\\
         b^{-1}\sigma(t)&1
     \end{pmatrix}\in D^\times$, then $d\cdot0=\dfrac{d}{1}=t$.
     \end{proof}

     Let $\chi_1,\chi_2$ be normalized characters of $E^\times$. Denote by $\pi(\chi_1,\chi_2)$ the principal series representation of $\GL_2(E)$ induced from $(\chi_1\boxtimes\chi_2)$, that is $\pi(\chi_1,\chi_2)=\Ind_{B_2(E)}^G(\chi_1\boxtimes\chi_2)$.\\

     \begin{lem}\label{3.2}
         Let $\pi=\pi(\chi_1,\chi_2)$ be a principal series representation of $\GL_2(E)$. Then $\pi$ is distinguished by $D^\times$ if and only if $\chi_1\chi_2^\sigma=\boldsymbol{1}$.
         Moreover,we have
         \begin{align*}
            \mathrm{dim}\ \Hom_{D^\times}(\pi,\boldsymbol{1})=\begin{cases}
                 1,\  if\ \chi_1\chi_2^\sigma=\boldsymbol{1},\\
                 0,\  otherwise.
             \end{cases}
         \end{align*}
     \end{lem}

     \begin{proof}
         By Lemma \ref{3.1} and Mackey Theory, we have $\Res_{D^\times}^G\Ind_{B_2(E)}^G(\chi_1\boxtimes\chi_2)=\Ind_{D^\times\cap B_2(E)}^{D^\times}\chi_1\boxtimes \chi_2$. Then
         \begin{align*}
             \Hom_{D^\times}(\pi,\boldsymbol{1})=\Hom_{D^\times}(\Ind_{D^\times\cap B_2(E)}^{D^\times}\chi_1\boxtimes\chi_2,\boldsymbol{1})=\Hom_{D^\times\cap B_2(E)}(\chi_1\boxtimes\chi_2,\boldsymbol{1}).
         \end{align*}
         Because $B_2(E)\cap D^\times=\left\{\begin{pmatrix}
             \alpha&\\
             &\sigma(\alpha)
         \end{pmatrix}|\alpha\in E,\Nm_{E/F}(\alpha)\neq 0\right\}\simeq E^\times$, one has $$\mathrm{dim}\ \Hom_{D^\times}(\pi,\boldsymbol{1})=\begin{cases}
                 1,\  if\ \chi_1\chi_2^\sigma=\boldsymbol{1},\\
                 0,\  otherwise.
             \end{cases}$$
             This finishes the proof.
     \end{proof}

 Assume $\ell\neq p$. 
 We will divide the proof of Theorem \ref{mainthm1.1} into two cases: $\ell$ is odd and $\ell=2$.
     \subsection{The case $\ell\neq2$.}
    \ 
    \newline
\indent By Lemma \ref{3.2}, we have the results for all irreducible principal series. We refer to \cite{vig1989} for the following well-known facts about $\pi(\chi_1,\chi_2)$. The principal series representation $\pi(\chi_1,\chi_2)$ is reducible if and only if $\chi_1\chi_2^{-1}=\nu^{\pm1}$, where $\nu$ is the absolute valuation of $E^\times$. Let $\chi$ be a character of $E^\times $. We only need to study the irreducible subquotients of $\pi(\chi\nu^{1/2},\chi\nu^{-1/2})$. 

If $q_E^2\not\equiv1\pmod \ell$, then $\pi(\chi\nu^{1/2},\chi\nu^{-1/2})$ has length 2 with irreducible subquotients $\chi$ and $\mathrm{St}_\chi$. The exact sequence of $\GL_2(E)$-module is:
 \begin{align*}
0\longrightarrow\mathrm{St}_\chi\longrightarrow\pi(\chi\nu^{1/2},\chi\nu^{-1/2})\longrightarrow\chi\circ\det\longrightarrow0
\end{align*}
Taking the functor $\Hom_{D^\times}(-,\boldsymbol{1})$, we have:
\begin{align*}\label{exactsequence St}
             0\longrightarrow&\Hom_{D^\times}(\chi\circ\det,\boldsymbol{1})\longrightarrow\Hom_{D^\times}(\pi(\chi\nu^{1/2},\chi\nu^{-1/2}),\boldsymbol{1})\\&\longrightarrow\Hom_{D^\times}(\mathrm{St}_\chi,\boldsymbol{1})\longrightarrow\Ext^1_{D^\times}(\chi\circ\det,\boldsymbol{1})\tag{4-1}
\end{align*}

If $q_E\equiv-1\pmod \ell$, then $\pi(\chi\nu^{1/2},\chi\nu^{-1/2})$ has length 3 with irreducible subquotients $\chi$, $\chi\nu_2$ and $\mathrm{Sp}_\chi$. The exact sequences of $\GL_2(E)$-module are:
\begin{align*}
             0\longrightarrow J_\chi\longrightarrow \pi(\chi\nu^{1/2},\chi\nu^{-1/2})\longrightarrow \chi\circ\det\longrightarrow0
\end{align*}
 and
\begin{align*}
             0\longrightarrow(\chi\circ\det)\nu_2\longrightarrow J_\chi\longrightarrow \mathrm{Sp}_\chi\longrightarrow0
\end{align*}
Taking the functor $\Hom_{D^\times}(-,\boldsymbol{1})$, we have:
\begin{align*}\label{f34}
             0\longrightarrow\Hom_{D^\times}(\chi\circ\det,\boldsymbol{1})&\longrightarrow\Hom_{D^\times}(\pi(\chi\nu^{1/2},\chi\nu^{-1/2}),\boldsymbol{1})\\
             \longrightarrow \Hom_{D^\times}(J_\chi,\boldsymbol{1})&\longrightarrow\Ext^1_{D^\times}(\chi\circ\det,\boldsymbol{1})\longrightarrow\Ext_{D^\times}^1(\pi(\chi\nu^{1/2},\chi\nu^{-1/2}),\boldsymbol{1})  \tag{4-2}
\end{align*}
and
 \begin{align*}\label{f35}
             0\longrightarrow\Hom_{D^\times}(\Sp_\chi,\boldsymbol{1})\longrightarrow\Hom_{D^\times}(J_\chi,\boldsymbol{1})\longrightarrow\Hom_{D^\times}((\chi\circ\det)\nu_2,\boldsymbol{1}).\tag{4-3}
\end{align*}

If $q_E\equiv1\pmod\ell$, then $\nu$ is trivial and $\pi(\chi,\chi)$ is the direct sum of $\chi$ and $\mathrm{St}_\chi$. Here $\mathrm{Sp}_\chi$ or $\mathrm{St}_\chi$ is the unique generic subquotient representation. 

We first introduce a lemma about $\Ext_{D^\times}^1(\pi(\chi\nu^{1/2},\chi\nu^{-1/2}),\boldsymbol{1})$.

\begin{lem}\label{l4.2}
    Suppose $\ell\neq p$, $\pi=\pi(\chi\nu^{1/2},\chi\nu^{-1/2})$, $\chi$ is a character of $E^\times$ and $\chi^2|_{F^\times}=\boldsymbol{1}$. Then
    $\pi(\chi\nu^{1/2},\chi\nu^{-1/2})|_{D^\times}$ can be seen as a representation of $PD^\times$ and
    \begin{align*}
       \dim\ \Ext_{PD^\times}^1(\pi,\boldsymbol{1})=\begin{cases}
        1,& \mathrm{if\ }E/F\ \mathrm{is\ unramified}, \chi\chi^\sigma=\boldsymbol{1}\ \mathrm{and}\  \ell|q+1;\\
        1,& \mathrm{if\ }E/F\ \mathrm{is\ ramified\ and}\ \ell=2;\\
        0,& \mathrm{otherwise}.
    \end{cases}
    \end{align*}
    In particular, the lemma also holds when $p=2$.
\end{lem}
\begin{proof}
    Followed by Corollary \ref{exth1}, since $PD^\times/U_1\simeq D_{2(q+1)}$ is a finite group, we have
    \begin{align*}
        \Ext_{PD^\times}^1(\pi,\boldsymbol{1})\cong \mathrm{H}^1(D_{2(q+1)},(\pi^{U_1})^\vee).
    \end{align*}
    where $(\pi^{U_1})^\vee$ denotes the linear dual space of $\pi^{U_1}$.
    Then we will first determine $\pi^{U_1}$. According to the proof of Lemma \ref{3.2} and $\chi^2|_{F^\times}=\boldsymbol{1}$, $\Res_{D^\times}^G\pi=\Ind_{E^\times}^{D^\times}\chi\chi^\sigma=\Ind_{E^\times/F^\times}^{PD^\times}\chi\chi^\sigma$. Therefore, $\pi^{U_1}=(\Res_{D^\times}^G\pi)^{U_1}=(\Ind_{E^\times/F^\times}^{PD^\times}\chi\chi^\sigma)^{U_1}$. If $\pi^{U_1}=0$, $\Ext_{PD^\times}^1(\pi,\boldsymbol{1})=0$. Now we suppose $\pi^{U_1}\neq0$. The elements in $\pi^{U_1}$ are of the form:
    \begin{align*}
        \{f:PD^\times\rightarrow\overline{\boldsymbol{F}_\ell}|f(tg)=\chi\chi^\sigma(t) f(g),f(gu)=f(g),\forall t\in E^\times/F^\times,u\in U_1,g\in PD^\times\}.
    \end{align*}
    Choosing $g_0\in PD^\times$ such that $f(g_0)\neq0$, $\forall t\in(E^\times/F^\times)\cap U_1$, consider $f(tg_0)$:
    \begin{align*}
        f(tg_0)&=\chi\chi^\sigma(t)f(g_0)\\
        f(tg_0)&=f(g_0(g_0^{-1}tg_0))=f(g_0).
    \end{align*}
    The second equality holds since $U_1\lhd PD^\times$ is a normal subgroup. Then $\chi\chi^\sigma$ is trivial on $(E^\times/F^\times)\cap U_1$ which implies $\chi\chi^\sigma$ can be seen as a character of $(E^\times/F^\times)/({E^\times/F^\times\cap U_1})\simeq(E^\times/F^\times)U_1/U_1$. Therefore 
\begin{align*}
    \pi^{U_1}&=(\Ind_{E^\times/F^\times}^{PD^\times}\chi\chi^\sigma)^{U_1}\\
&=\Ind_{(E^\times/F^\times)U_1/U_1}^{PD^\times/U_1}\chi\chi^\sigma\\
&=\begin{cases}
        \Ind_{C_{q+1}}^{D_{2(q+1)}}\chi\chi^\sigma,\   \mathrm{if}\ E/F\ \mathrm{is\ unramified};\\
        \Ind_{C_2}^{D_{2(q+1)}}\chi\chi^\sigma,\   \mathrm{if}\ E/F\ \mathrm{is\ ramified}.
    \end{cases}
\end{align*}
By the Shapiro Lemma,
   \begin{align*}
   \mathrm{H}^1(D_{2(q+1)},(\pi^{U_1})^\vee)=\begin{cases}
           \mathrm{H}^1(D_{2(q+1)},\Ind_{C_{q+1}}^{D_{2(q+1)}}(\chi\chi^\sigma)^{-1})\cong\mathrm{H}^1(C_{q+1},(\chi\chi^\sigma)^{-1}),&   \mathrm{if}\ E/F\ \mathrm{is\ unramified};\\
           \mathrm{H}^1(D_{2(q+1)},\Ind_{C_{2}}^{D_{2(q+1)}}(\chi\chi^\sigma)^{-1})\cong\mathrm{H}^1(C_{2},(\chi\chi^\sigma)^{-1}),&   \mathrm{if}\ E/F\ \mathrm{is\ ramified}.
       \end{cases}
   \end{align*}
    \cite[Chapter VIII, Section 4]{serre1979local} shows that for $\mathrm{H}^1(G,A)$ with a cyclic group $G$ and a nontrivial one-dimensional module $A$, $\mathrm{H}^1(G,A)=0$. 
   
   Therefore, when $E/F$ is unramified, assuming $\chi\chi^\sigma=\boldsymbol{1}$, we have $\mathrm{H}^1(C_{q+1},(\chi\chi^\sigma)^{-1})\cong \Hom(C_{q+1},\overline{\boldsymbol{F}_\ell})$ and its dimension is one if and only if $\ell|q+1$. 
   
   If $E/F$ is ramified, suppose $\chi\chi^\sigma=\boldsymbol{1}$. We have $\mathrm{H}^1(C_{2},(\chi\chi^\sigma)^{-1})\cong \Hom(C_{2},\overline{\boldsymbol{F}_\ell})$ and its dimension is one if and only if $\ell=2$. Moreover, in this case, since $\chi^2|_{F^\times}=\boldsymbol{1}$, $\chi|_{F^\times}=\boldsymbol{1}$ and then $\chi\chi^\sigma=\boldsymbol{1}$ is automatically satisfied.
   
This proves the lemma.
\end{proof}

Now we need to separate them into two subcases: $p$ is odd and $p=2$.

\subsubsection{}We first study the case $p\neq 2$ and consider the distinction problem for $\mathrm{Sp}_\chi$ or $\mathrm{St}_\chi$.
\begin{prop}\label{prop:Cq+1}
    Let $\chi$ be a quadratic character of $F^\times$. Then $\chi\circ\det$ is trivial on $k_D^\times/k_F^\times\cong C_{q+1}$ if and only if $\chi$ is trivial on $k_F^\times$, where $k_D$(resp. $k_F$) denotes the residue field of $D$(resp. $F$).
\end{prop}
\begin{proof}
     Recall that from the proof of Lemma \ref{3.3}(2), $C_{q+1} \cong k_D^\times / k_F^\times$, 
where $k_D \cong \mathbb F_{q^2}$ and $ k_F \cong \mathbb F_q $. 
The determinant map $ \det : D^\times \to F^\times $ is given by the reduced norm
$\mathrm{Nrd} : D^\times \to F^\times,$
which induces a surjection on residue fields $k_D^\times \twoheadrightarrow k_F^\times.$ Hence the restriction of $ \chi \circ \det $ to $ C_{q+1} $ is trivial if and only if 
$\chi$ is trivial on $k_F^\times$.
\end{proof}

\begin{coro}\label{coro:ramified character}
    If $\chi$ is a quadratic ramified character of $F^\times$, then $\chi\circ\det$ is not trivial on $C_{q+1}$. 
\end{coro}
\begin{proof}
Suppose that \(\chi \circ \det\) is trivial on \(C_{q+1}\). By Proposition \ref{prop:Cq+1}, \(\chi\) is trivial on \(k_F^\times\). Moreover, since \(\chi\) is a quadratic character and \(1+\mathfrak{p}_F\) is a pro-\(p\) group with \(p \neq 2\), \(\chi\) must be trivial on \(1+\mathfrak{p}_F\). Consequently, \(\chi\) is trivial on \(\mathcal{O}_F^\times\), which contradicts the assumption that \(\chi\) is a ramified character of \(F^\times\).
\end{proof}
In the following theorem, we will give some criterions for when $\St_\chi\ \mathrm{or}\ \Sp_\chi$ is distinguished.

    \begin{thm}\label{thm1}
        Let $\chi$ be a character of $E^\times$ and $p$ be an odd prime.
        \begin{enumerate}
            \item If $q_E^2\not\equiv1(\mathrm{mod}\ \ell)$, then $\mathrm{St}_\chi$ is distinguished if and only if $\chi|_{F^\times}=\omega_{E/F}$.
            \item If $q_E\equiv1(\mathrm{mod}\ \ell)$, then $\mathrm{St}_\chi$ is distinguished if and only if $\chi|_{F^\times}=\omega_{E/F}$.
            \item If $q_E\equiv-1(\mathrm{mod}\ \ell)$, then $\mathrm{Sp}_\chi$ is distinguished if and only if one of the following conditions holds:
            \begin{enumerate}
                \item $\chi|_{F^\times}=\omega_{E/F}\ \mathrm{and}\ E/F\ \mathrm{is\ ramified}$.
                \item $\chi|_{F^\times}\ \mathrm{is\ a\ nontrivial\  quadratic\ unramified\ character\ and}\ E/F\ \mathrm{is\ ramified}$.
            \end{enumerate}
        \end{enumerate}
        Moreover, if $\mathrm{St}_\chi$ or $\mathrm{Sp}_\chi$ is distinguished, the dimension of $\Hom_{D^\times}(\mathrm{St}_\chi,\boldsymbol{1})$ or $\Hom_{D^\times}(\mathrm{Sp}_\chi,\boldsymbol{1})$ is one.
    \end{thm}
    \begin{proof}
    \
    \newline
\begin{enumerate}
    \item Assume $q_E^2\not\equiv1(\mathrm{mod}\ \ell)$. If $ \Hom_{D^\times}(\pi(\chi\nu^{1/2},\chi\nu^{-1/2}),\boldsymbol{1})\neq 0$, by Lemma \ref{3.2}, we have $\chi\chi^{\sigma}=\boldsymbol{1}$. Then $\chi|_{F^\times}\circ \Nm_{E/F}=\boldsymbol{1}$ which implies $\chi|_{F^\times}=\boldsymbol{1}_{F^\times}$ or $\omega_{E/F}$  (see \cite[Lemma 4.2]{cui2024}). At this time, the central character of $\pi$ restricting to the center $Z(D^\times)$ is trivial and so  $\mathrm{dim}\ \Hom_{PD^\times}(\pi,\boldsymbol{1})=\mathrm{dim}\ \Hom_{D^\times}(\pi,\boldsymbol{1})=1$.
         \begin{itemize}[leftmargin=0pt]
             \item If $\chi|_{F^\times}=\boldsymbol{1}$, $\chi\circ\det$ factors through $D^\times\rightarrow PD^\times$. Taking the functor $\Hom_{PD^\times}(-,\boldsymbol{1})$, we have
         \begin{align*}
             0\longrightarrow&\Hom_{PD^\times}(\chi\circ\det,\boldsymbol{1})\longrightarrow\Hom_{PD^\times}(\pi(\chi\nu^{1/2},\chi\nu^{-1/2}),\boldsymbol{1})\\&\longrightarrow\Hom_{PD^\times}(\mathrm{St}_\chi,\boldsymbol{1})\longrightarrow\Ext^1_{PD^\times}(\chi\circ\det,\boldsymbol{1})
         \end{align*}
         and dim $\Hom_{PD^\times}(\chi\circ\det,\boldsymbol{1})=1$.
         Due to Lemma \ref{3.3}(3) that $\Ext^1_{PD^\times}(\chi\circ\det,\boldsymbol{1})=0$, then $\mathrm{St}_\chi$ is not distinguished.
         \item If $\chi|_{F^\times}=\omega_{E/F}$, we have $\Hom_{D^\times}(\chi\circ\det,\boldsymbol{1})=0$. Since $q_E^2\not\equiv1\pmod\ell,\ell\nmid (q_E+1)$. By Lemma \ref{3.3}(2), $\Ext^1_{D^\times}(\chi\circ\det,\boldsymbol{1})=0$. Therefore, by the exact sequence \eqref{exactsequence St}, $\mathrm{St}_\chi$ is  distinguished and $\mathrm{dim}\ \Hom_{D^\times}(\mathrm{St}_\chi,\boldsymbol{1})=1$.
         \end{itemize}

If $\Hom_{D^\times}(\pi(\chi\nu^{1/2},\chi\nu^{-1/2})=0$, we can consider another exact sequence
\begin{align*}
    0\longrightarrow\chi\circ\det\longrightarrow\pi(\chi\nu^{-1/2},\chi\nu^{1/2})\longrightarrow\mathrm{St}_\chi\longrightarrow0
\end{align*}
Taking the functor $\Hom_{D^\times}(-,\boldsymbol{1})$, we have $\Hom_{D^\times}(\mathrm{St}_\chi,\boldsymbol{1})=0$ i.e. $\mathrm{St}_\chi$ is not distinguished.
\item 
Suppose $q_E\equiv1(\mathrm{mod\ \ell)}$. Then $\nu$ is trivial and
\[\dim\Hom_{D^\times}(\pi,\boldsymbol{1})=\dim\Hom_{D^\times}(\chi,\boldsymbol{1})+\dim\Hom_{D^\times}(\mathrm{St}_\chi,\boldsymbol{1}). \]
         Suppose that $ \Hom_{D^\times}(\pi,\boldsymbol{1})\neq 0$. If $\chi|_{F^\times}=\boldsymbol{1}$, then $\mathrm{St}_\chi$ is not distinguished. If $\chi|_{F^\times}=\omega_{E/F}$, 
       then $\dim\Hom_{D^\times}(\mathrm{St}_\chi,\boldsymbol{1})=1$. Moreover, if $\pi$ is not distinguished, $\mathrm{St}_\chi$ is also not distinguished.
         \\
\item 
         Suppose that $q_E\equiv-1(\mathrm{mod}\ \ell)$. When $\Hom_{D^\times}(\pi(\chi\nu^{1/2},\chi\nu^{-1/2}),\boldsymbol{1})\neq0$, assuming $\chi|_{F^\times}=\boldsymbol{1}$, by a similar argument to the case $q_E^2\not\equiv1(\mathrm{mod}\ \ell)$, $\Hom_{D^\times}(J_\chi,\boldsymbol{1})=0$. Then  $\Hom_{D^\times}(\mathrm{Sp}_\chi,\boldsymbol{1})=0$, i.e. $\mathrm{Sp}_\chi$ is not distinguished by the exact sequence \eqref{f35}.
         
         Suppose that $\chi|_{F^\times}=\omega_{E/F}$. If $E/F$ is unramified, $q_E=q^2$ and $\ell|q^2+1$ then $\ell\nmid q+1$. Otherwise $\ell|2q$, contradicting $\ell\neq2$ and  $\ell\neq p$. By Lemma \ref{3.3}(2), $\Ext^1_{D^\times}(\chi\circ\det,\boldsymbol{1})=0$. Therefore, followed by the exact sequence \eqref{f34}, $\mathrm{dim}\ \Hom_{D^\times}(J_\chi,\boldsymbol{1})=1$. For the exact sequence \eqref{f35}), consider $(\chi\circ\det)\nu_2|_{D^\times}$. Denote by $\nu_F$ the absolute value of $F^\times$ and $\nu=\nu_F\circ \Nm_{E/F}$. Given any $ g\in D^\times$,$$(\chi\circ\det)\nu_2(g)=\chi(\det(g))|\det(g)|_E=\chi(\det(g))|\det(g)|_F^2=(\chi\cdot\nu_F^2)\circ\det(g).$$ Since $q^2\equiv-1(\mathrm{mod}\ \ell),\ \nu_F^2=\omega_{E/F}$. In this case $\chi|_{F^\times}\cdot\nu_F^2=\omega_{E/F}^2=\boldsymbol{1}_{F^\times}$. Therefore, $\Hom_{D^\times}((\chi\circ\det)\nu_2,\boldsymbol{1})=\Hom_{D^\times}(\boldsymbol{1},\boldsymbol{1})$ with dimension 1. Then we will prove $\Hom_{D^\times}(J_\chi,\boldsymbol{1})\longrightarrow \Hom_{D^\times}(\boldsymbol{1},\boldsymbol{1})$ is surjective.

         Actually, since the central character restricting to the center $Z(D^\times)$ is trivial, we can rewrite the exact sequence \eqref{f35} in the category of smooth representations of $PD^\times$:
         \begin{align*}
             0\longrightarrow\Hom_{PD^\times}(\Sp_\chi,\boldsymbol{1})\longrightarrow\Hom_{PD^\times}(J_\chi,\boldsymbol{1})\longrightarrow\Hom_{PD^\times}(\boldsymbol{1},\boldsymbol{1})\longrightarrow\Ext_{PD^\times}^1(\Sp_\chi,\boldsymbol{1}).
         \end{align*}
         Followed by Proposition \ref{extfix}, $$\Ext^1_{PD^\times}(\Sp_\chi,\boldsymbol{1})\cong\Ext_{D_{2(q+1)}}^1((\Sp_\chi)^{U_1},\boldsymbol{1})$$ where $U_1=(1+\mathfrak{p}_D)F^\times/F^\times$. Since $\ell\nmid 2(q+1)$, Maschke's Theorem implies the group algebra $\overline{\boldsymbol{F}_\ell}[D_{2(q+1)}]$ is semisimple. Then $\Ext_{D_{2(q+1)}}^1((\Sp_\chi)^{U_1},\boldsymbol{1})=0$ i.e. $\Hom_{D^\times}(J_\chi,\boldsymbol{1})\longrightarrow \Hom_{D^\times}(\boldsymbol{1},\boldsymbol{1})$ is surjective.

         Therefore, $\dim \Hom_{PD^\times}(\Sp_\chi, \boldsymbol{1})=0$, i.e., $\Sp_\chi$ is not distinguished.\\

         If $E/F$ is ramified, $q_E=q$ and so $\ell|q+1$. In this case, the quadratic character $\chi|_{F^\times}=\omega_{E/F}$ is ramified. Then Corollary \ref{coro:ramified character} implies $\chi\circ\det|_{C_{q+1}}$ is not trivial. Followed by Lemma \ref{3.3}(2), one has $\Ext_{D^\times}(\chi\circ\det,\boldsymbol{1})=0$. Now, by the exact sequence \eqref{f34}, $\mathrm{dim}\ \Hom_{D^\times}(J_\chi,\boldsymbol{1})=1$. Thanks to the fact that $q=q_E\equiv-1(\mathrm{mod}\ \ell)$, $q^2\equiv1(\mathrm{mod}\ \ell)$ then $\nu_F^2$ is trivial and 
\begin{align*}
    \Hom_{D^\times}((\chi\circ\det)\nu_2,\boldsymbol{1})=\Hom_{D^\times}((\chi\cdot\nu_F^2)\circ\det,\boldsymbol{1})=\Hom_{D^\times}(\chi\circ\det,\boldsymbol{1})=0.
\end{align*}
By the exact sequence \eqref{f35},  $\mathrm{dim}\ \Hom_{D^\times}(\mathrm{Sp}_\chi,\boldsymbol{1})=1$ and so $\mathrm{Sp_\chi}$ is distinguished.\\

If $\pi$ is not distinguished by $D^\times$, Lemma \ref{3.2} implies $\chi\chi^\sigma\neq\boldsymbol{1}$ i.e., $\chi|_{F^\times}\not\in\{\boldsymbol{1},\omega_{E/F}\}$.

Suppose $E/F$ is unramified. Then $q_E=q^2$ and $\ell\nmid\ q+1$. By Lemma \ref{3.3}(2), $\Ext^1_{D^\times}(\chi\circ\det,\boldsymbol{1})=0$ which implies $\Hom_{D^\times}(J_\chi,\boldsymbol{1})=0$. Then $\Sp_\chi$ is not distinguished followed by the exact sequence \eqref{f35}.

Suppose $E/F$ is ramified. Then $q=q_E$ and $\ell|q+1$. The central character of $\Sp_\chi$ is $\chi^2$. If $\chi^2|_{F^\times}\neq\boldsymbol{1}$, $\Sp_\chi$ must be not distinguished. Thus we only need to consider the case $\chi^2|_{F^\times}=\boldsymbol{1}$ with $\chi|_{F^\times}\not\in\{\boldsymbol{1},\omega_{E/F}\}$. Since the central character restricting to the center $Z(D^\times)$ is trivial, we can reduce the exact sequence \eqref{f34} to those representations of $PD^\times$:
\begin{align*}
             0\longrightarrow\Hom_{PD^\times}(\chi\circ\det,\boldsymbol{1})&\longrightarrow\Hom_{PD^\times}(\pi(\chi\nu^{1/2},\chi\nu^{-1/2}),\boldsymbol{1})\\
             \longrightarrow \Hom_{PD^\times}(J_\chi,\boldsymbol{1})&\longrightarrow\Ext^1_{PD^\times}(\chi\circ\det,\boldsymbol{1})\longrightarrow\Ext_{PD^\times}^1(\pi(\chi\nu^{1/2},\chi\nu^{-1/2}),\boldsymbol{1}).  
\end{align*}
If $\chi|_{F^\times}$ is a ramified quadratic character, $\Ext_{PD^\times}(\chi\circ\det,\boldsymbol{1})=\Ext_{D^\times}(\chi\circ\det,\boldsymbol{1})=0$ by Corollary \ref{coro:ramified character} and Lemma \ref{3.3}(2). Then $\Sp_\chi$ is not distinguished. 

If $\chi|_{F^\times}$ is an unramified quadratic character, then Proposition \ref{prop:Cq+1} shows that $\chi\circ\det$ is trivial on $C_{q+1}$. Thus dim $\Ext_{PD^\times}(\chi\circ\det,\boldsymbol{1})=1$ by Lemma \ref{3.3}(2) and Lemma \ref{l4.2} implies $\Ext_{PD^\times}^1(\pi,\boldsymbol{1})=0$. Therefore, $$\dim\ \Hom_{PD^\times}(J_\chi,\boldsymbol{1})=\dim\ \Ext_{PD^\times}^1(\chi\circ\det,\boldsymbol{1})=1.$$ Followed by
the exact sequence \eqref{f35} and the fact that $\Hom_{D^\times}((\chi\circ\det)\nu_2,\boldsymbol{1})=0$, $\Sp_\chi$ is distinguished and $\dim\ \Hom_{D^\times}(\Sp_\chi,\boldsymbol{1})=1$.

Therefore, when $\pi$ is not distinguished, $\Sp_\chi$ is distinguished if and only if $E/F$ is ramified and $\chi|_{F^\times}$ is an unramified quadratic character.
\end{enumerate}

\noindent This finishes the proof.
 \end{proof}

\subsubsection{}Now, we focus on the case $p=2$.

For Lemma \ref{3.3}, the case $p=2$ does not affect the argument that if $\chi^2\neq\boldsymbol{1}$, then $\Ext_{D^\times}^1(\boldsymbol{1},\chi\circ\det)=0$. Moreover, according to our argument and \cite{rie70}, the structures of $D^\times$, $PD^\times$ and their filtrations also remain unchanged and so the exact sequences \eqref{f2} and \eqref{f3} still hold.

When $\chi=\boldsymbol{1}$, $\mathrm{H}^1(PD^\times,\boldsymbol{1})=\Hom(D_{2(q+1)}^{ab},\overline{\boldsymbol{F}_\ell})$. At this time, $q+1$ is odd and so $D_{2(q+1)}^{ab}\cong C_2$. Since $\ell\neq p$, $\Hom(D^{ab}_{2(q+1},\overline{\boldsymbol{F}_\ell}) $ is also zero.

\begin{lem}\label{p2quadratic}
    In the case $p=2$, when $\chi^2|_{F^\times}=\boldsymbol{1}$ but $\chi|_{F^\times}\neq\boldsymbol{1}$, we have $$\mathrm{dim}\ \Ext_{D^\times}(\boldsymbol{1},\chi\circ\det)=\begin{cases}
        1,\ \mathrm{if}\ \chi|_{F^\times}\ \mathrm{is\ unramified\ and}\  \ell|q+1;\\
        0,\ \mathrm{otherwise}.
    \end{cases}$$
\end{lem}
\begin{proof}
    Thanks to the exact sequence \eqref{f1} and $\mathrm{H}^1(F^\times,(\chi\circ\det))^{PD^\times}=0$ in Lemma \ref{3.3}, we have $\mathrm{H}^1(D^\times,\chi\circ\det)=\mathrm{H}^1(PD^\times,\chi\circ \det)$. Consider the exact sequence \eqref{f2}:
    \begin{align*}
            0\longrightarrow \mathrm{H}^1(D_{2(q+1)},(\chi\circ\det)^{U_1})\longrightarrow\mathrm{H}^1(PD^\times,\chi\circ\det)\longrightarrow \mathrm{H}^1(U_1,\chi\circ\det)^{D_{2(q+1)}}.
        \end{align*}
    where $U_1=(1+\mathfrak{p}_D)F^\times/F^\times$.
        
    When $\chi|_{F^\times}$ is ramified, since $\chi|_{F^\times}$ is quadratic and $p=2$, $\chi|_{F^\times}$ is trivial on $\mathcal{O}_F^\times/(1+\mathfrak{p}_F)\cong C_{q-1}$. Then $\chi|_{F^\times}$ is not trivial on $1+\mathfrak{p}_F$; otherwise $\chi|_{F^\times}$ is trivial on $\mathcal{O}_F^\times$ contradicting the fact that $\chi$ is ramified. The map $\det|_{1+\mathfrak{p}_D}:1+\mathfrak{p}_D\longrightarrow1+\mathfrak{p}_F$ is surjective and so $\chi\circ\det$ is not trivial on $U_1$. Then $(\chi\circ\det)^{U_1}=0$ and $\mathrm{H}^1(D_{2(q+1)},(\chi\circ\det)^{U_1})=0$.

    For $\mathrm{H}^1(U_1,\chi\circ\det)$, $U_1$ is a pro-2 group and $\chi\circ\det$ is a discrete $\ell$-torsion module and $\ell\neq 2$. Followed by Corollary \ref{copro-p}, $\mathrm{H}^1(U_1,\chi\circ\det)=0$. Therefore, when $\chi$ is ramified, we have $$\mathrm{H}^1(D^\times,\chi\circ\det)=\mathrm{H}^1(PD^\times,\chi\circ\det)=0.$$
    When $\chi|_{F^\times}$ is unramified, $\chi|_{\mathcal{O}_F^\times}=\boldsymbol{1}$ and $\chi(\varpi_F)=-1$. Then $\chi\circ\det$ is trivial on $\mathcal{O}_D^\times$ especially on $U_1$. In this case, 
    \begin{align*}
        (\chi\circ\det)^{U_1}=\chi\circ\det, \ \mathrm{H}^1(U_1,\chi\circ\det)=\Hom(U_1,\overline{\boldsymbol{F}_\ell})=0\ \mathrm{since}\ \ell\neq p.
    \end{align*}
     Now, we likewise only need to compute $\mathrm{H}^1(D_{2(q+1)},\chi\circ\det)$. If $p=2$, $q+1$ is odd and so $\chi\circ\det$ must be trivial over $C_{q+1}$. The remaining argument is the same as the case $\chi\circ\det|_{C_{q+1}}=\boldsymbol{1}$ in Lemma \ref{3.3}. Therefore, when $\chi$ is unramified, dim $\Ext_{D^\times}^1(\boldsymbol{1},\chi\circ\det)=\begin{cases}
         1,\ \ell|\ q+1;\\
         0,\ \mathrm{otherwise}.
     \end{cases}$
     
     In a short summary, $\mathrm{dim}\ \Ext_{D^\times}^1(\boldsymbol{1},\chi\circ\det)=\begin{cases}
        1,& \mathrm{if}\ \chi|_{F^\times}\ \mathrm{is\ unramified\ and}\  \ell|\ q+1;\\
        0,& \mathrm{otherwise}.
    \end{cases}$.
\end{proof}

For Theorem \ref{thm1}, according to the argument as before, we only need to focus on the case when $q_E\equiv-1(\mathrm{mod}\ \ell)$.
\begin{prop}\label{pSp2}
    When $p=2$, $q_E\equiv-1(\mathrm{mod}\ \ell)$, $\mathrm{Sp}_\chi$ is distinguished if and only if one of the following conditions holds:
    \begin{enumerate}
        \item $\chi|_{F^\times}=\omega_{E/F}$ and $E/F$ is ramified;
        \item $\chi|_{F^\times}$ is a nontrivial quadratic unramified character and $E/F$ is ramified.
    \end{enumerate}
    Moreover, if $\Sp_\chi$ is distinguished, the dimension of $\Hom_{D^\times}(\Sp_\chi,\boldsymbol{1})$ is 1.
\end{prop}
\begin{proof}
Suppose $\pi$ is distinguished. Then the central character restricting the center $Z(D^\times)$ is trivial. We can reduce to the case for $PD^\times$. If $\chi|_{F^\times}$ is trivial, then $\Ext_{PD^\times}^1(\chi\circ\det,\boldsymbol{1})=0$ and $\Hom_{PD^\times}(J_\chi,\boldsymbol{1})=0$ implies that $\Sp_\chi$ is not distinguished due to \eqref{f35}. 

If $\chi|_{F^\times}=\omega_{E/F}$, supposing $E/F$ is unramified, then $\ell|\ q^2+1$ shows that $\ell\nmid q+1$. Followed by Lemma \ref{p2quadratic}, $\Ext_{D^\times}(\chi\circ\det,\boldsymbol{1})=0$. Then $$\dim \Hom_{D^\times}(J_\chi,\boldsymbol{1})=\dim \Hom_{D^\times}((\chi\circ\det)\nu_2,\boldsymbol{1})=1.$$ Since $\ell\nmid 2(q+1)$, $\Ext_{PD^\times}^1(\Sp_\chi,\boldsymbol{1})\cong\Ext_{D_{2(q+1)}}^1((\Sp_\chi)^{U_1},\boldsymbol{1})=0$ by Maschke's Theorem. The map $\Hom_{D^\times}(J_\chi,\boldsymbol{1})\longrightarrow\Hom_{D^\times}(\boldsymbol{1},\boldsymbol{1})$ is surjective and so $\Sp_\chi$ is not distinguished.

Now suppose $E/F$ is ramified. Then $\omega_{E/F}$ is ramified. Using Lemma \ref{p2quadratic}, $\Ext_{D^\times}^1(\chi\circ\det,\boldsymbol{1})=0$. Followed by the exact sequences \eqref{f34} and \eqref{f35}, we have the isomorphisms:
\begin{align*}
    \Hom_{D^\times}(\pi,\boldsymbol{1})\cong \Hom_{D^\times}(J_\chi,\boldsymbol{1})\ \mathrm{and}\ \Hom_{D^\times}(\mathrm{Sp}_\chi,\boldsymbol{1})\cong \Hom_{D^\times}(J_\chi,\boldsymbol{1})
\end{align*}
Since $\dim \Hom_{D^\times}(\pi,\boldsymbol{1})= 1$, $\Hom_{D^\times}(\mathrm{Sp}_\chi,\boldsymbol{1})\neq 0$ i.e., $\mathrm{Sp}_\chi$ is distinguished and its dimension is 1.\\

Suppose $\pi$ is not distinguished. We only consider the case $\chi^2|_{F^\times}=\boldsymbol{1}$ with $\chi|_{F^\times}\not\in\{\boldsymbol{1},\omega_{E/F}\}$. If $E/F$ is unramified, then $\ell\nmid q+1$. Lemma \ref{p2quadratic} implies $\Ext_{D^\times}(\chi\circ\det,\boldsymbol{1})=0$. Then $\Sp_\chi$ is not distinguished since $\Hom_{D^\times}(J_\chi,\boldsymbol{1})=0$. If $E/F$ is ramified, Lemma \ref{l4.2} implies $\Ext_{PD^\times}^1(\pi,\boldsymbol{1})=0$ and $\Hom_{PD^\times}((\chi\circ\det)\nu_2,\boldsymbol{1})=\Hom_{D^\times}((\chi\circ\det)\nu_2,\boldsymbol{1})=0$. Thus we have 
\begin{align*}
    \Hom_{D^\times}(\Sp_\chi,\boldsymbol{1})\simeq\Ext_{PD^\times}^1(\chi\circ\det,\boldsymbol{1})
\end{align*}
by the exact sequences \eqref{f34} and \eqref{f35}. Therefore $\Sp_\chi$ is distinguished if and only if $\chi|_{F^\times}$ is unramified, which is followed by Lemma \ref{p2quadratic}. 
This finished the proof.
\end{proof}

\subsection{The case $\ell=2$.}
    \ 
    \newline
\indent
This section is the most difficult part in the whole paper. We will transfer the distinction problem for the cuspidal representation $\Sp_\chi$ to the finite group case. Write down all maps in the long exact sequence \eqref{f34}  explicitly and do the computation carefully. 
When $\ell=2$, the equality $\GL_2(E)=B_2(E)D^\times$ still holds. Therefore, by Lemma \ref{3.2}, if $\pi=\pi(\chi_1,\chi_2)$ is a normalized principal series representation of $\GL_2(E)$, $\pi$ is distinguished by $D^\times$ if and only if $\chi_1\chi_2^\sigma=\boldsymbol{1}$ and 
     \begin{align*}
             \mathrm{dim}\ \Hom_{D^\times}(\pi,\boldsymbol{1})=\begin{cases}
                 1,\  if\ \chi_1\chi_2^\sigma=\boldsymbol{1},\\
                 0,\  otherwise.
             \end{cases}
         \end{align*}

     In this case, we always have $q_E\equiv1(\mathrm{mod}\ \ell)$, which means $\nu$ is always trivial. Referring to \cite{vig1989}, $\pi(\chi_1,\chi_2)$ is reducible if and only if $\chi_1=\chi_2$. Moreover, $\pi(\chi,\chi)$ has length 3 and the irreducible subquotients are $\chi$, $\chi$ and $\mathrm{Sp}_\chi$, where $\mathrm{Sp}_\chi$ is an infinite-dimensional representation.
    \begin{thm}\label{thm41}
    Assume $\ell=2$.
     The irreducible representation   $\Sp_\chi$ is distinguished by $D^\times$ if and only if $\chi|_{F^\times}=\boldsymbol{1}$, $E/F$ is ramified and $q\equiv 3\pmod 4$.
    \end{thm}
        \begin{proof}
            We have two exact sequences
            \begin{align*}\label{expi}
                0\longrightarrow M\longrightarrow \pi(\chi,\chi)&\longrightarrow\chi\circ\det\longrightarrow0\tag{4-4}
            \end{align*}
            and
            \begin{align*}\label{f41}
                0\longrightarrow\chi\circ\det\longrightarrow &M\longrightarrow\Sp_\chi\longrightarrow0\tag{4-5}
            \end{align*}
            
        If $\pi(\chi,\chi)$ is not distinguished, we must have $\chi|_{F^\times}\neq\boldsymbol{1}$; otherwise  $\chi\chi^\sigma=\chi|_{F^\times}\circ \Nm_{E/F}$ is trivial, which will imply that $\pi(\chi,\chi)$ is distinguished by $D^\times$. Then $\chi^2|_{F^\times}\neq\boldsymbol{1}$ since $\ell=2$ and so
         $\Sp_\chi$ is not distinguished.

        Suppose that $\pi(\chi,\chi)$ is distinguished, i.e., $\chi\chi^\sigma=\boldsymbol{1}$. Then $\chi^2|_{F^\times}=\boldsymbol{1}$ which means $\chi|_{F^\times}=\boldsymbol{1}$.
        Each element in $\pi(\chi,\chi)$ is of the form 
        \begin{align*}
            \{f:G\rightarrow \overline{\boldsymbol{F}_2}\ |f\  \mathrm{is \ smooth\ and}\ f(bg)=\chi(a)\chi(d)f(g),\forall b=\begin{pmatrix}
                a&x\\
                &d
            \end{pmatrix},g\in G\}
        \end{align*}
        Therefore, $\forall z\in Z(D^\times)=F^\times, z\cdot f(g)=f(g\cdot zI_2)=f(zI_2\cdot g)=\chi^2(z)f(g)=f(g)$. We have that $z$ acts trivially on $\pi(\chi,\chi)$. Then $\pi(\chi,\chi)$ can be viewed as a representation of $PD^\times$ and $\Hom_{D^\times}(\pi(\chi,\chi),\boldsymbol{1})=\Hom_{PD^\times}(\pi(\chi,\chi),\boldsymbol{1})$, so are $M$ and $\chi\circ\det$. Taking the funcor $\Hom_{PD^\times}(-,\boldsymbol{1})$, we get the long exact sequence
        \begin{align*}\label{f42}
            0&\longrightarrow\Hom_{PD^\times}(\chi\circ\det,\boldsymbol{1})\stackrel{f_1}{\longrightarrow}\Hom_{PD^\times}(\pi(\chi,\chi),\boldsymbol{1})\stackrel{f_2}{\longrightarrow}\Hom_{PD^\times}(M,\boldsymbol{1})\\
            &\stackrel{f_3}{\longrightarrow} \Ext_{PD^\times}^1(\chi\circ\det,\boldsymbol{1})\stackrel{f_4}{\longrightarrow}\Ext_{PD^\times}^1(\pi(\chi,\chi),\boldsymbol{1})\tag{4-6}
        \end{align*}
        where $f_i$ denotes the corresponding map.
In this case,  $$\dim\Hom_{PD^\times}(\chi\circ\det,\boldsymbol{1})= \dim \Hom_{PD^\times}(\pi(\chi,\chi),\boldsymbol{1})=1$$ and $f_1$ is injective. Therefore $f_1$ is an isomorphism. By the exactness of \eqref{f42}, $f_2$ is a zero map and $f_3$ is injective. According to the dimension formula of the exact sequence,
\begin{align*} \label{f43}
    \mathrm{dim}\ \Ext_{PD^\times}^1(\chi\circ\det,\boldsymbol{1})=\mathrm{dim\ Ker}f_4+\mathrm{dim\ Im}f_4=\mathrm{dim}\ \Hom_{PD^\times}(M,\boldsymbol{1})+ \mathrm{dim\ Im}f_4\tag{4-7}
\end{align*}
For the dimension of $\Ext_{PD^\times}^1(\chi\circ\det,\boldsymbol{1})$, by the exact sequence \eqref{f2} , \begin{align*}
    \Ext_{PD^\times}^1(\chi\circ\det,\boldsymbol{1})=\mathrm{H}^1(PD^\times,\chi\circ\det)=\mathrm{H}^1(D_{2(q+1)},\boldsymbol{1})
\end{align*}
 since $\ell\neq p$ and $\chi|_{F^\times}=\boldsymbol{1}$. Then the dimension is 2. For $\Ext^1_{PD^\times}(\pi(\chi,\chi),\boldsymbol{1})$, Lemma \ref{l4.2} shows that $\dim\ \Ext^1_{PD^\times}(\pi(\chi,\chi),\boldsymbol{1})=1$ since $\ell=2$ and $\ell|\ (q+1)$.

Therefore, we just need to determine the dimension of $\mathrm{Im}$ $f_4$.

\begin{lem}\label{l4.3}
Let $\pi=\pi(\chi,\chi)$.
   The map $f_4:\Ext^1_{PD^\times}(\chi\circ\det,\boldsymbol{1})\rightarrow \Ext_{PD^\times}^1(\pi,\boldsymbol{1})$ is surjective. In other words, $\mathrm{dim\ Im}\ f_4=1$.
\end{lem}
\begin{proof}
Recall that $U_1$ is the pro-$p$ subgroup of $PD^\times$ such that $PD^\times/U_1\cong D_{2(q+1)}$.
Note that $$\mathrm{H}^1(D_{2(q+1)},(\pi^{U_1})^\vee)\cong\begin{cases}
        \mathrm{H}^1(C_{q+1},\boldsymbol{1}),&   \mathrm{if}\ E/F\ \mathrm{is\ unramified};\\
        \mathrm{H}^1(C_2,\boldsymbol{1}),&   \mathrm{if}\ E/F\ \mathrm{is\ ramified}.\\
    \end{cases}$$
We have $\Ext_{PD^\times}^1(\chi\circ\det,\boldsymbol{1})\cong\mathrm{H}^1(D_{2(q+1)},\boldsymbol{1})$ and $\Ext_{PD^\times}^1(\pi,\boldsymbol{1})\cong \mathrm{H}^1(D_{2(q+1)},(\pi^{U_1})^\vee)$. Then  
\begin{align*}
    f_4: \mathrm{H}^1(D_{2(q+1)},\boldsymbol{1})\longrightarrow\mathrm{H}^1(D_{2(q+1)},(\pi^{U_1})^\vee)
\end{align*}
is the restriction map by \cite[Example 1.27(b)]{milneCFT}.

    If $E/F$ is unramified, taking $\psi\in\Hom(C_{q+1},\overline{\boldsymbol{F}_2})=\mathrm{H}^1(C_{q+1},\boldsymbol{1})$, we can define $\varphi\in\Hom(D_{2(q+1)},\overline{\boldsymbol{F}_2})=\Ext_{PD^\times}^1(\chi\circ\det,\boldsymbol{1})$ as follows:
    \begin{align*}
        \varphi:D_{2(q+1)}\longrightarrow\overline{\boldsymbol{F}_2},\quad \varphi(r)=\psi(r)\ \mathrm{and}\ \varphi(s)=0
    \end{align*}
    where $r$ and $s$ are generators of the dihedral group $D_{2(q+1)}$ satisfying $r^{q+1}=1$ and $s^2=1$.

    If $E/F$ is ramified, taking $\psi\in\Hom(C_{2},\overline{\boldsymbol{F}_2})=\mathrm{H}^1(C_2,\boldsymbol{1})$, we can define $\varphi\in\Hom(C_2\times C_2,\overline{\boldsymbol{F}_2})=\Ext_{PD^\times}^1(\chi\circ\det,\boldsymbol{1})$ as follows:
    \begin{align*}
        \varphi:C_2\times C_2\longrightarrow\overline{\boldsymbol{F}_2},\quad \varphi(a,b)=\psi(a),\forall a,b\in C_2.
    \end{align*}
    Therefore $f_4$ is a surjective map and $\dim \mathrm{Im} f_4=1$.
\end{proof}
Let us back to the proof of Theorem \ref{thm41}.
Followed by the equality \eqref{f43} and Lemma \ref{l4.3},  $\dim\Hom_{PD^\times}(M,1)$=1. Taking the functor $\Hom_{PD^\times}(-,\boldsymbol{1})$ on the exact sequence \eqref{f41}:
\begin{align*}\label{Mhom}
    0\longrightarrow \Hom_{PD^\times}(\mathrm{Sp}_\chi,\boldsymbol{1})\longrightarrow \Hom_{PD^\times}(M,\boldsymbol{1})\stackrel{f_5}{
    \longrightarrow
    }\Hom_{PD^\times}(\chi\circ\det,\boldsymbol{1})\tag{4-8}
\end{align*}
Since $\dim \Hom_{PD^\times}(M, \mathbf{1}) = \dim \Hom_{PD^\times}(\chi \circ \det, \mathbf{1}) = 1$, we only need to figure out whether $f_5$ is a zero map or an isomorphism.

\begin{lem}\label{f5zero}
The map $f_5$ is a zero map if and only if $E/F$ is ramified and $q\equiv3\pmod4$.
\end{lem}
\begin{proof}
    Since $U_1=(1+\mathfrak{p}_D)F^\times/F^\times$ is a normal open subgroup of $PD^\times$, for any smooth representation $V$, there exists an isomorphism 
    \begin{align*}
        \Hom_{PD^\times}(V,\boldsymbol{1})\simeq\Hom_{PD^\times/U_1}(V^{U_1},\boldsymbol{1})
    \end{align*}
    where $PD^\times/U_1\simeq D_{2(q+1)}$. Then we can rewrite the exact sequence \eqref{Mhom}:
    \begin{align*}
    0\longrightarrow \Hom_{D_{2(q+1)}}(\mathrm{Sp}_\chi^{U_1},\boldsymbol{1})\longrightarrow \Hom_{D_{2(q+1)}}(M^{U_1},\boldsymbol{1})\stackrel{f_5}{
    \longrightarrow
    }\Hom_{D_{2(q+1)}}(\boldsymbol{1},\boldsymbol{1})\simeq\overline{\boldsymbol{F}_2}.
\end{align*}

To determine $M^{U_1}$, taking the $U_1-$fixed functor on the exact sequence \eqref{expi}, we have
\begin{align*}
    0\longrightarrow M^{U_1}\longrightarrow \pi^{U_1}\stackrel{\varepsilon}\longrightarrow\overline{\boldsymbol{F}_2}\longrightarrow0
\end{align*}
since  $\chi|_{F^\times}=\boldsymbol{1}$. Here $\varepsilon$ is the quotient map which will be defined explicitly later. 

Followed by the structure of $\pi^{U_1}$ in Lemma \ref{l4.2}, since $\chi\chi^\sigma=\boldsymbol{1}$, $\pi^{U_1}=(\Ind_{E^\times/F^\times}^{PD^\times}\boldsymbol{1})^{U_1}$. We just need to consider the \((E^\times/F^\times, U_1)\)-double cosets of \(PD^\times\). Denote $$ PD^\times/U_1 \simeq D_{2(q+1)} = \langle r, s \mid r^{q+1} = s^2 = 1, srs = r^{-1} \rangle.$$ When \(E/F\) is unramified, there are 2 representatives \(\{1, s\}\) in the double cosets and \(f \in \pi^{U_1}\) is uniquely determined by \(f(1), f(s)\); when \(E/F\) is ramified, there are \(q+1\) representatives \(\{1, r, \cdots, r^q\}\) and \(f \in \pi^{U_1}\) is uniquely determined by \(f(1), f(r), \cdots, f(r^q)\). 
In a short summary, $$\pi^{U_1}\simeq \begin{cases}
\overline{\boldsymbol{F}_2}^2,&\text{if } E/F \text{ is unramified,} \\ \overline{\boldsymbol{F}_2}^{q+1},&\text{if } E/F \text{ is ramified,}
\end{cases}$$ via the map
\[ f \longmapsto \begin{cases}
  (f(1), f(s)),&\text{if } E/F \text{ is unramified;} \\
 (f(1), f(r), \cdots, f(r^q)),&\text{if } E/F \text{ is ramified.}
\end{cases}\]

Now we consider the $PD^\times/U_1\simeq D_{2(q+1)}$-equivariant map $\varepsilon:\pi^{U_1}\longrightarrow\overline{\boldsymbol{F}_2}$. When $E/F$ is unramified, $s\cdot f(1)=f(s), s\cdot f(s)=f(1)$ and $r\in C_{q+1}$ acts trivially and so $\varepsilon(f)=f(1)+f(s)$. When $E/F$ is ramified, $r\cdot f(r^j)=f(r^{j+1\pmod {q+1}}),s\cdot f(r^j)=f(r^{-j\pmod{q+1}})$. Thus $\varepsilon(f)=f(1)+f(r)+\cdots+f(r^q)$. Since $M^{U_1}$ is the kernel of $\varepsilon$, fix a suitable basis of $\pi^{U_1}$. Then we have
\begin{align*}
    M^{U_1}\simeq\begin{cases}
        \{(a_0,a_1)\in\overline{\boldsymbol{F}_2}^2|\ a_0+a_1=0\},&  E/F\ \textrm{is\ unramified};\\
        \{(a_0,\cdots,a_q)\in\overline{\boldsymbol{F}_2}^{q+1}|\ a_0+\cdots+a_q=0\},&E/F\ \textrm{is\ ramified}.
    \end{cases}
\end{align*}

Followed by the exact sequence \eqref{f41}, denote by $\iota$ the embedding map $(\chi\circ\det)^{U_1}\hookrightarrow M^{U_1}$. Then for any $\lambda\in \Hom_{D_{2(q+1)}}(M^{U_1},\boldsymbol{1})$,
since $\chi\circ\det$ is trivial on $U_1$, we have
\begin{align*}
    f_5(\lambda)=\lambda\circ\iota=\begin{cases}
    \lambda(1,1),& E/F \mathrm{\ is\  unramified};\\
    \lambda(1,\cdots,1),& E/F\ \mathrm{is\  ramified}.
\end{cases}
\end{align*}

 When $E/F$ is unramified, dim $M^{U_1}=1$ and so $f_5(\lambda)=\lambda(1,1)=1$, not a zero map. 
 
 When $E/F$ is ramified, we define 
\begin{align*}
    \lambda(a_0,\cdots,a_q)=\displaystyle\sum_{i=0}^q ia_i.
\end{align*}
Then we will verify that it is $D_{2(q+1)}$-invariant. Here all the indices are taken modulo \( q+1 \).

For the action of $r$:
\begin{align*}
    \lambda(r\cdot(a_0,\cdots,a_q))=\displaystyle\sum_{i=0}^{q}ia_{i+1}=\displaystyle\sum_{i=0}^q(i-1)a_i=\displaystyle\sum_{i=0}^q ia_i-\displaystyle\sum_{i=0}^q a_i=\displaystyle\sum_{i=0}^q i a_i.
\end{align*}
The last equality holds since $(a_0,\cdots,a_q)\in M^{U_1}$.

For the action of $s$:
\begin{align*}
    \lambda(s\cdot(a_0,\cdots,a_q))=\displaystyle\sum_{i=0}^qia_{-i}=\displaystyle\sum_{i=0}^q(-i)a_i=\displaystyle\sum_{i=0}^qia_i.
\end{align*}
The last equality holds since $\ell=2$. Therefore, $\lambda$ is $D_{2(q+1)}$-invariant. Moreover, since $\Hom_{D_{2(q+1)}}(M^{U_1},\boldsymbol{1})$ is one-dimensional, it is spanned by $\lambda$. Then 
\begin{align*}
    \lambda(1,\cdots,1)=\displaystyle\sum_{i=0}^q i\equiv\begin{cases}
        1\pmod 2, &\mathrm{when}\  q\equiv1\pmod4;\\
        0\pmod 2, &\mathrm{when}\  q\equiv3\pmod4.
    \end{cases}
\end{align*}
In a short summary, $f_5$ is a zero map if and only if $E/F$ is ramified and $q\equiv3\pmod4$. 
\end{proof}
Followed by Lemma \ref{f5zero}, $\Sp_\chi$ is distinguished if and only if $f_5$ is a zero map. It is equivalent to the conditions that $\chi|_{F^\times}=\boldsymbol{1}$, $E/F$ is ramified and $q\equiv3\pmod4$. Moreover,
\begin{align*}
    \mathrm{dim}\ \Hom_{D^\times}(\Sp_\chi,\boldsymbol{1})=\begin{cases}
        1,\ E/F\ \mathrm{is\ ramified\ and}\ q\equiv3\pmod4;\\
        0,\ \mathrm{otherwise}.
    \end{cases}
\end{align*}
This finishes the proof.
\end{proof}

    \section{The Supercuspidal representations}
In this section, we will deal with the case for the supercuspidal representations.
    \begin{thm}\label{c3.1}
        Let $\pi$ be an irreducible supercuspidal representation of $\GL_2(E)$. Then $\pi$ is $D^\times$-distinguished if and only if $\pi$ is $\GL_2(F)$-distinguished. (Here $\ell$ is arbitary except $\ell=p$)
    \end{thm}
    \begin{proof}
    Followed by \cite[Theorem 3.4]{KM20}, for either $H=D^\times$ or $\GL_2(F)$, an integral supercuspidal $\overline{\mathbb Q_{\ell}}$-representation
$\widetilde\pi$ satisfies that if $\widetilde\pi$ is $H$-distinguished, then $r_\ell(\widetilde\pi)$ is $H-$distinguished where $r_\ell(\widetilde\pi)$ is the reduction modulo $\ell$ of $\widetilde\pi$.

Suppose $\pi$ is $D^\times$-distinguished. Since $D^\times$ is a closed subgroup of $\GL_2(E)$, \cite[Theorem 3.4]{cui2024} implies that if $\pi$ is $D^\times$-distinguished. Then there exists a $\overline{\mathbb{Q}_\ell}$-lift $\widetilde{\pi}$ of $\pi$ such that it is supercuspidal and distinguished by $D^\times$. Note that $\tilde{\pi}$ is $D^\times$-distinguished if and only if $\tilde{\pi}$ is $\GL_2(F)$-distinguished by \cite[Theorem C]{prased1992}. Then $\pi$ is $\GL_2(F)$-distinguished since $\pi$ is the reduction modulo $\ell$ of $\widetilde\pi$.

Conversely, suppose $\pi$ is $\GL_2(F)$-distinguished. Because $\GL_2(F)$ is also closed in $\GL_2(E)$, by a similar argument as before,  we have $\pi$ is $D^\times$-distinguished.
\end{proof}

    \begin{thm}\label{thm42}
        Let $\pi$ be a $\sigma$-selfdual supercuspidal representation of $\GL_2(E)$. Assume that $\ell\neq2$. Then the followings are equivalent:
        \begin{enumerate}
            \item $\pi$ is $\GL_2(F)$-distinguished;
            \item $\pi$ is not $\omega_{E/F}$-distinguished;
            \item Langlands parameter $\varphi_\pi$ is conjugate-orthogonal.
        \end{enumerate}
        In particular, a $\sigma$-selfdual supercuspidal representation is either $\GL_2(F)$-distinguished or $(\GL_2(F),\omega_{E/F})$-distinguished, but not both.
    \end{thm}
    \begin{proof}
        This is a special case of \cite[Theorem 3.14 and Proposition 3.15]{cui2024}.
    \end{proof}
       
Theorem \ref{c3.1} reduces the $D^\times$-distinction problem for supercuspidal representations to the $\GL_2(F)$-distinction problem.  Theorem \ref{thm42} tells us that for $\sigma$-selfdual supercuspidal representations it suffices to decide between distinction and $\omega_{E/F}$-distinction. A complete local criterion is provided by the type‑theoretic construction of S\'echerre in \cite{vincent2019}. We now recall the relevant invariants and the resulting classification, which becomes particularly simple in the $\GL_2$ case.\\

Let $\pi$ be a $\sigma$-selfdual supercuspidal representation of
$\mathrm{GL}_2(E)$. Let $m\in\{1,2\}$ be its relative degree, and let
$T/T_0$ be the quadratic extension attached to $\pi$ as in
\cite[Proposition 5.14]{vincent2019}.
Following \cite[Definition 10.4]{vincent2019}, let
$(\mathcal D/\mathcal D_0,\delta_0)$ be the quadratic extension
and character attached to $\pi$.
Here we use $\mathcal D$ to distinguish this field from the
quaternion division algebra $D$.

The character $\delta_0$ is a quadratic character: either the trivial character or, when $\ell \neq 2$,
the local class field theory character $\omega_{\mathcal D/\mathcal D_0}$ associated to the
quadratic extension $\mathcal D/\mathcal D_0$.
By \cite[Proposition 10.5]{vincent2019}, the pair
$(\mathcal D/\mathcal D_0,\delta_0)$ is uniquely determined
by $\pi$ up to $F$-isomorphism.
We call $\delta_0$ the distinction character of $\pi$.

\begin{thm}\label{sccriterion}
Let $\pi$ be a $\sigma$-selfdual supercuspidal representation of $\GL_2(E)$,
with relative degree $m \in \{1,2\}$ and distinction character $\delta_0$.
\begin{enumerate}
\item If $\ell = 2$, then $\pi$ is always $D^\times$-distinguished.
\item If $\ell \neq 2$ and $p\neq2$, then $\pi$ is $D^\times$-distinguished if and only if 
\begin{align*}
    \begin{cases}
    \delta_0\ \mathrm{is\ the\ trivial\ character}, &\mbox{when}\ m=1;\\\
    \delta_0=\omega_{\mathcal{D}/\mathcal{D}_0},& \mbox{when}\ m=2.
\end{cases}
\end{align*}
\end{enumerate}
\end{thm}

\begin{proof}
By Theorem \ref{c3.1}, $\pi$ is $D^\times$-distinguished if and only if $\pi$ is $\GL_2(F)$-distinguished.
\begin{enumerate}
    \item 
The case $\ell=2$ is \cite[Theorem 10.8]{vincent2019}.  
\item 
Assume $\ell\neq2$.
Suppose $m=1$, $T/T_0$ and $\mathcal D/\mathcal D_0$ must be ramified or unramified at the same time. Followed by \cite[Theorem 10.9,(3)]{vincent2019}, $\pi$ is distinguished if and only if $\delta_0$ is trivial. 

Now suppose $m=2$. Followed by \cite[Proposition 9.8]{vincent2019}, $T/T_0$ is ramified; otherwise $m$ will be odd. Moreover, \cite[Remark 10.10]{vincent2019} implies $\mathcal{D}/\mathcal{D}_0$ is unramified. Then \cite[Theorem 10.9,(2)]{vincent2019} shows that $\pi$ is distinguished if and only if $\delta_0=\omega_{\mathcal D/\mathcal D_0}$.
\end{enumerate}
This finishes the proof.
\end{proof}

\section{$\GL_2(F)$-distinguished representations}
In this section, we deal with the cases remained in the work of \cite{cui2024}, namely the classification of $\mathrm{GL}_2(F)$-distinguished principal series representations of $\mathrm{GL}_2(E)$ when $\ell = 2$. 

Assume $\ell=2$ throughout this section. Then $\ell\neq p$, $q_F\equiv1\pmod{\ell}$ and $q_E\equiv1\pmod{\ell}$. Referring to \cite{vig1989}, $\pi(\chi_1,\chi_2)$ is reducible if and only if $\chi_1=\chi_2$. Moreover, $\pi(\chi,\chi)$ has length 3 and the irreducible subquotients are $\chi$, $\chi$ and $\mathrm{Sp}_\chi$, where $\mathrm{Sp}_\chi$ is an infinite-dimensional representation.
The double coset decomposition of $\GL_2(E)$ is
\[
\mathrm{GL}_2(E) = B_2(E)\mathrm{GL}_2(F) \sqcup B_2(E)\eta\,\mathrm{GL}_2(F),
\]
where $\eta = \begin{pmatrix} 1 & \delta \\ 1 & -\delta \end{pmatrix}$ and $E = F[\delta]$. By the geometric lemma, one has
\[
0 \to \mathrm{c-Ind}_{E^\times}^{\GL_2(F)} \chi_1 \chi_2^\sigma \to \pi \to \Ind_{B_2(F)}^{\GL_2(F)} \delta_{B_2(E)}^{1/2} \delta_{B_2(F)}^{-1/2} (\chi_1 \boxtimes \chi_2) \to 0
\]
with \( \delta_{B_2(E)}^{1/2} = \delta_{B_2(F)} \). Applying the functor \( \Hom_{\GL_2(F)}(-, \mathbf{1}_2) \), we obtain the long exact sequence
\[\label{f51}
0 \to \Hom_{F^\times \times F^\times} (\chi_1 \boxtimes \chi_2, \mathbf{1}) \to \Hom_{\GL_2(F)}(\pi, \mathbf{1}_2) \to \Hom_{E^\times} (\chi_1 \chi_2^\sigma, \mathbf{1}) \to \Ext_{F^\times \times F^\times}^{1} (\chi_1 \boxtimes \chi_2, \mathbf{1}).\tag{6-1}
\]

\begin{prop}\label{pil=2}Assume $\ell=2$. Then the followings are equivalent:
\begin{enumerate}
    \item $\pi=\pi(\chi_1,\chi_2)$ is $\GL_2(F)$-distinguished;
    \item  $\chi_1\chi_2^\sigma=\boldsymbol{1}$ or $\chi_1|_{F^\times}=\chi_2|_{F^\times}=\boldsymbol{1}$ with $\chi_1\neq\chi_2$.
\end{enumerate}
\end{prop}
\begin{proof}
    Followed by the exact sequence \eqref{f51}:
    \begin{enumerate}
        \item If $\chi_1|_{F^\times}\neq 1$ and $\chi_2|_{F^\times}\neq 1$, by Lemma \ref{l51}, $\Ext_{F^\times\times F^\times}^1(\chi_1\boxtimes\chi_2,\boldsymbol{1})=0$, then $\Hom_{\GL_2(F)}(\pi,\boldsymbol{1})\cong\Hom_{E^\times}(\chi_1\chi_2^\sigma,\boldsymbol{1})$. Therefore, $\pi$ is $\GL_2(F)$-distinguished if and only if $\chi_1\chi_2^\sigma=\boldsymbol{1}$.
        \item If just one of the two characters $\chi_1|_{F^\times},\chi_2|_{F^\times}$ is $\boldsymbol{1}$, then $\Hom_{F^\times \times F^\times} (\chi_1 \boxtimes \chi_2, \mathbf{1})=\Hom_{E^\times} (\chi_1 \chi_2^\sigma, \mathbf{1})=0$, and so $\pi$ is not $\GL_2(F)$-distinguished.
        \item If $\chi_1|_{F^\times}=\chi_2|_{F^\times}=\boldsymbol{1}$, dim $\Hom_{F^\times \times F^\times} (\chi_1 \boxtimes \chi_2, \mathbf{1})=1$, and so $\Hom_{\GL_2(F)}(\pi,\boldsymbol{1})\neq 0$ i.e. $\pi$ is $\GL_2(F)$-distinguished.
    \end{enumerate}
    This finishes the proof.
\end{proof}
Having established the criterion for the principal series $\pi(\chi_1,\chi_2)$ to be $\GL_2(F)$-distinguished in Proposition \ref{pil=2}, we now turn to the remaining irreducible component that appears when the principal series is reducible, namely the special representation $\mathrm{Sp}_\chi$.
\begin{thm}\label{spl=2} Assume $\ell=2$. Then
    $\mathrm{Sp}_\chi$ is $\GL_2(F)$-distinguished if and only if $\chi|_{F^\times}=\boldsymbol{1}$.
\end{thm}
\begin{proof}
   Recall that $\mathrm{Sp}_\chi$ is a subquotient representation of $\pi(\chi,\chi)$ and that it is cuspidal. According to \cite[Theorem 1.3]{vin2026}, $\mathrm{Sp}_\chi$ is $\GL_2(F)$-distinguished if and only if $\mathrm{Sp}_\chi$ is $\sigma$-selfdual. 
   
   Note that
        $(\mathrm{Sp}_\chi)^\vee\simeq\mathrm{Sp}_{\chi^{-1}}$ and  $(\mathrm{Sp}_\chi)^\sigma\simeq\mathrm{Sp}_{\chi^\sigma}$. 
         When $\mathrm{Sp}_\chi$ is $\sigma$-selfdual, comparing the central characters, we have $(\chi^{-1})^2=(\chi^\sigma)^2$ then $\chi|_{F^\times}=\boldsymbol{1}$ since $\ell=2$. Conversely, when $\chi|_{F^\times}=\boldsymbol{1}$, $\Sp_\chi$ is $\sigma$-selfdual and so it is $\GL_2(F)$-distinguished.
\end{proof}

\bibliographystyle{amsalpha}
\bibliography{mira}
\enlargethispage{3\baselineskip}

\end{document}